\documentclass[preprint,11pt]{elsarticle}
\usepackage{amsfonts}

 \usepackage{bbm}
 \usepackage{amsmath}
 \usepackage{amssymb}
 \usepackage{mathrsfs}
 \usepackage[all]{xy}
 \usepackage{amsthm,amsmath,amsfonts,amssymb}
\usepackage{graphicx}

\biboptions{numbers,sort&compress}

\newtheorem{thm}{Theorem}[section]
\newtheorem{lem}{Lemma}[section]

\newtheorem{cor}{Corollary}[section]

\newtheorem{Acknw}{Acknowledgement}

\theoremstyle{definition}

\def\beq{\begin{equation}}
\def\eeq{\end{equation}}
\def\ben{\begin{equation*}}
\def\een{\end{equation*}}
\def \b{\begin}
\def\e{\end}

\begin{document}
\numberwithin{equation}{section}

\begin{frontmatter}
\title{Cegrell's classes and a variational approach for the quaternionic Monge-Amp\`{e}re equation\tnoteref{mytitlenote}}
\tnotetext[mytitlenote]{This work is supported by National Nature Science Foundation in China (No. 11401390; No. 11571305).}

\author{Dongrui Wan\corref{mycorrespondingauthor}}
\cortext[mycorrespondingauthor]{Corresponding author}
\ead{wandongrui@szu.edu.cn,+86 755 2653 8953,
Room 415, Science and Technology Building, Shenzhen University}

\address{College of Mathematics and Statistics, Shenzhen University, Room 415 Science and Technology Building, Shenzhen, 518060, PR. China}

\begin{abstract}In this paper, we introduce finite energy classes of quaternionic  plurisubharmonic functions of Cegrell type and study the quaternionic Monge-Amp\`{e}re operator on these classes on quaternionic hyperconvex domains of $\Bbb H^{n}$. We extend the domain of definition of quaternionic Monge-Amp\`{e}re operator to some Cegrell's classes, the functions of which are not necessarily bounded. We show that integration by parts and comparison principle are valid on some classes. Moreover, we use the variational method to solve the quaternionic Monge-Amp\`{e}re equations when the right hand side is a positive measure of finite energy.\end{abstract}

\begin{keyword} Monge-Amp\`{e}re operator; quaternionic plurisubharmonic function;
cegrell's class; variational approach 
\end{keyword}

\end{frontmatter}

\section{Introduction}\par
The problem of extending the domain of definition of complex Monge-Amp\`ere operator $(dd^c\cdot)^n$ on some class of plurisubharmonic functions was studied by many analysts in the past three decades. In the fundamental papers \cite{bed,beddirichlet} Bedford and Taylor showed that this operator is well defined on the class of locally bounded plurisubharmonic functions and they studied the Dirichlet problem for the complex Monge-Amp\`ere equation with continuous data in a strictly pseudoconvex domain of $\Bbb C^n$. There are several methods to define the complex Monge-Amp\`ere operator as a positive measure for some unbounded plurisubharmonic functions \cite{kiselman,demailly85,sibony85,cegrell1986,bed1990}. In particular, Cegrell \cite{cegrell1998,cegrell2004} introduced several finite energy classes of plurisubharmonic functions on which the complex Monge-Amp\`ere operator can be defined and has important properties. These Cegrell's classes have fruitful applications in recent years.

The aim of this paper is to introduce these Cegrell's finite energy classes for quaternionic plurisubharmonic functions, and to extend the domain of definition of quaternionic Monge-Amp\`ere operator to some classes of unbounded plurisubharmonic functions on quaternionic hyperconvex domains of $\Bbb H^{n}$. Moreover, we establish some comparison principle and convergence results for quaternionic Monge-Amp\`ere operator on these classes. Finally, using the variational method we solve the following quaternionic Monge-Amp\`ere equations on quaternionic hyperconvex domains of $\Bbb H^{n}$, $$(\triangle \varphi)^n=\mu, $$ 
where $(\triangle \varphi)^n$ denotes the quaternionic Monge-Amp\`ere operator of $\varphi$ and $\mu$ is a positive Radon measure in $\Omega$.

The quaternionic Monge-Amp\`ere equation has attracted considerable studies in recent years. The quaternionic version of Calabi-Yau conjecture, which has applications in superstring theory, has been solved recently on compact manifolds with a flat hyperk\"ahler metric \cite{alesker9}. Relating to this problem, it is interesting to study the quaternionic Monge-Amp\`ere equations \cite{alesker2,alesker6,alesker7,wan6,wang1}.

Denote by $PSH(\Omega)$ the class
of quaternionic plurisubharmonic functions on $\Omega$ and by $PSH^-(\Omega)$ the subclass of negative functions in $PSH(\Omega)$. Let $\Omega$ be a quaternionic hyperconvex domain of $n$-dimensional quaternionic space $\Bbb H^{n}$, this means that it is
open, bounded, connected and that there exists $\rho\in PSH^-(\Omega)$ such that
the set $\{q\in\Omega:\rho(q)<c\}$ is a relatively compact subset of
$\Omega$, for each $c\in(-\infty,0)$. Such function $\rho$ is called an
exhaustion function of $\Omega$. In this paper, unless otherwise specified, $\Omega$ will always be a quaternionic hyperconvex domain in $\Bbb H^{n}$.

In this paper, we study the following finite energy classes of quaternionic $PSH$ functions in $\Omega$. They are generalizations of Cegrell's classes for complex plurisubharmonic functions \cite{cegrell1998,cegrell2004}. Let $p\geq1$.
\small{\begin{itemize}
  \item $\mathcal {E}_0(\Omega)=\left\{\varphi\in PSH^-\cap L^\infty(\Omega):\lim_{q\rightarrow\partial\Omega}\varphi(q)=0,~\int_\Omega(\triangle\varphi)^n<+\infty\right\}$
  \item $\mathcal {F}(\Omega)=\left\{\varphi\in PSH^-(\Omega): \exists ~\mathcal {E}_0(\Omega)\ni \varphi_j\searrow \varphi,~\sup_{j\geq1}\int_\Omega(\triangle\varphi_j)^n<+\infty\right\}$
      \item $\mathcal {E}(\Omega)=\left\{\varphi\in PSH^-(\Omega): \exists ~\varphi_K\in \mathcal {F}(\Omega) ~\text{such that~} \varphi_K=\varphi~\text{on}~K, \forall K\Subset \Omega\right\}$
      \item $\mathcal {E}_p(\Omega)=\left\{\varphi\in PSH^-(\Omega): \exists ~\mathcal {E}_0(\Omega)\ni \varphi_j\searrow \varphi,~\sup_{j}\int_\Omega(-\varphi_j)^p(\triangle\varphi_j)^n<+\infty\right\}$
      \item$\mathcal {F}_p(\Omega)=\mathcal {E}_p(\Omega)\cap \mathcal {F}(\Omega)$
      \end{itemize}}As in the complex case, these classes satisfy
 $\mathcal {E}_0(\Omega)\subset \mathcal {F}(\Omega)\subset \mathcal {E}(\Omega)$ and $\mathcal {E}_0(\Omega)\subset \mathcal {F}_p(\Omega)\subset \mathcal {E}_p(\Omega)$.
 
 The paper is organized as follows. The preliminary results about the quaternionic plurisubharmonic function and quaternionic Monge-Amp\`{e}re operator are given in Section 2. In Section 3, first we establish a global approximation of negative $PSH$ functions on hyperconvex domains. The continuous functions in $\mathcal {E}_0(\Omega)$ can be considered as the test functions. By using the closed positive current results in \cite{wan4,wan7}, we prove that integration by parts is allowed in $\mathcal {E}_0(\Omega)$. Then we show that the quaternionic Monge-Amp\`ere operator is well defined for functions in $\mathcal {E}(\Omega)$ and $\mathcal {E}_p(\Omega)$, $p\geq1$. Integration by parts is also allowed in $\mathcal {F}(\Omega)$ and $\mathcal {E}_p(\Omega)$. Several inequalities including the energy estimate for quaternionic Monge-Amp\`{e}re operator are established on these classes. These preparations are needed in the proof of our main result. Finally in Section 4 we use the variational method to solve quaternionic Monge-Amp\`{e}re equations and prove the following main theorem.
 
  \b{thm}\label{t6.2'}Let $\mu$ be a positive Radon measure in $\Omega$ and $p\geq1$. Then we have $\mu=(\triangle \varphi)^n$ with $\varphi\in\mathcal{E}_p(\Omega)$ if and only if $\mathcal{E}_p(\Omega)\subset L^p(\Omega,\mu)$.
\e{thm}

The variational approach to solve complex Monge-Amp\`{e}re equations was done on hyperconvex domains of $\mathbb{C}^n$ \cite{cegrell2012}, and on compact K\"ahler manifolds \cite{guedj-variational}. Recently, Lu used the variational method to solve the degenerate complex Hessian equations \cite{lu,lu-degenerate}. 

We follow the methods from Cegrell \cite{cegrell1998,cegrell2004} for the complex Monge-Amp\`{e}re operator and from Lu \cite{lu,lu-degenerate} for the complex Hessian operator. The theory of quaternionic closed positive currents we established in \cite{wan3,wan4,wan7} allows us to treat the quaternionic Monge-Amp\`{e}re operator as an operator of divergence form, and so we can integrate by parts. Since this can avoid the inconvenience in using Moore determinant, we obtained several quaternionic pluripotential results in \cite{wan6,wan8}. All these preparations play key roles in this paper.

\section{Preliminaries}\par
A real valued function $f:\mathbb{H}^n\rightarrow \mathbb{R}$ is called quaternionic plurisubharmonic if it is upper semi-continuous and its restriction to any right quaternionic line is subharmonic (in the usual sense). Any quaternionic $PSH$ function is subharmonic (cf. \cite{alesker1,alesker4,alesker2} for more information about $PSH$ functions).
\b{pro}\label{p1.1} Let $\Omega$ be an open subset of $\mathbb{H}^n$.\\
(1). The family $PSH(\Omega)$ is a convex cone, i.e. if $\alpha,\beta$ are non-negative numbers and $u,v\in PSH(\Omega)$, then $\alpha u+\beta v\in PSH(\Omega)$; and $\max\{u,v\}\in PSH(\Omega)$.\\
(2). If $\Omega$ is connected and $\{u_j\}\subset PSH(\Omega)$ is a decreasing sequence, then $u=\lim_{j\rightarrow\infty}u_j\in PSH(\Omega)$ or $u\equiv-\infty$.\\
(3). Let $\{u_\alpha\}_{\alpha\in A}\subset PSH(\Omega)$ be such that its upper envelope $u=\sup_{\alpha\in A}u_\alpha$ is locally bounded above. Then the upper semicontinuous regularization $u^*\in PSH(\Omega)$.\\
(4). Let $\omega$ be a non-empty proper open subset of $\Omega$, $u\in PSH(\Omega),v\in PSH(\omega)$, and $\limsup_{q\rightarrow\zeta}v(q)\leq u(\zeta)$ for each $\zeta\in \partial\omega\cap\Omega$, then
$$w:=\left\{
      \begin{array}{ll}
        \max\{u,v\}, & \text{in}~\omega \\
        u, &  \text{in}~\Omega\backslash\omega
      \end{array}
    \right. \quad\in PSH(\Omega).$$\\
(5). If $u\in PSH(\Omega)$, then the standard regularization $u_\varepsilon:=u*\rho_\varepsilon$ is also $PSH$ in $\Omega_\varepsilon:=\{q\in \Omega:dist(q,\partial\Omega)>\varepsilon\}$, moreover, $u_\varepsilon\searrow u$ as $\varepsilon\rightarrow0$.\\
(6). If $\gamma(t)$ is a convex increasing function of a parameter $t\in\mathbb{R}$ and $u\in PSH$, then $\gamma\circ u\in PSH$.

\e{pro}
The quaternionic Monge-Amp\`{e}re operator is defined as the Moore determinant of the quaternionic Hessian of $C^2$ function $u$:
 \begin{equation}\label{det}det(u)=det\left[\frac{\partial^2u}{\partial q_j\partial \bar{q}_k}(q)\right].\end{equation}Alesker \cite{alesker1} extended the definition of quaternionic Monge-Amp\`{e}re operator to continuous $PSH$ functions, and introduced in \cite{alesker2} an operator in terms of the Baston operator $\triangle$. The $n$-th power of this operator is exactly the quaternionic Monge-Amp\`{e}re operator when the manifold is flat. The author and Wang introduced in \cite{wan3} two first-order differential operators $d_0$ and $d_1$ acting on the quaternionic version of differential forms. The behavior of $d_0,d_1$ and $\triangle=d_0d_1$ is very similar to $\partial ,\overline{\partial}$ and $\partial\overline{\partial}$ in several complex variables. The quaternionic Monge-Amp\`{e}re operator can be defined as $(\triangle u)^n$ and has a simple explicit expression, which is much more convenient to use than the previous definition (\ref{det}) by using  Moore determinant.

We use the conjugate embedding
\begin{equation*}\begin{aligned}\tau:\mathbb{H}^{n}\cong\mathbb{R}^{4n}&\hookrightarrow\mathbb{C}^{2n\times2},\\ (q_0,\ldots,q_{n-1})&\mapsto \textbf{z}=(z^{j\alpha})\in\mathbb{C}^{2n\times2},
\end{aligned}\end{equation*}
$j=0,1,\ldots,2n-1, ~\alpha=0 ,1 ,$ with
\begin{equation*}\label{2.2}\left(
                             \begin{array}{cc}
                               z^{(2l)0 } & z^{(2l)1 } \\
                               z^{(2l+1)0 } & z^{(2l+1)1 } \\
                             \end{array}
                           \right):=\left(
                                      \begin{array}{cc}
                                         x_{4l}-\textbf{i}x_{4l+1} & -x_{4l+2}+\textbf{i}x_{4l+3} \\
                                        x_{4l+2}+\textbf{i}x_{4l+3} & x_{4l}+\textbf{i}x_{4l+1} \\
                                      \end{array}
                                    \right),
\end{equation*}for $q_l=x_{4l}+\textbf{i}x_{4l+1}+\textbf{j}x_{4l+2}+\textbf{k}x_{4l+3}$, $l=0,1,\ldots,n-1$. Pulling back to the quaternionic
space $\mathbb{H}^n\cong\mathbb{R}^{4n}$ by the embedding above, we define on $\mathbb{R}^{4n}$ first-order differential operators $\nabla_{j\alpha}$ as following:
\begin{equation*}\label{2.4}\left(
                             \begin{array}{cc}
                               \nabla_{(2l)0 } & \nabla_{(2l)1 } \\
                               \nabla_{(2l+1)0 } & \nabla_{(2l+1)1 } \\
                             \end{array}
                           \right):=\left(
                                      \begin{array}{cc}
                                         \partial_{x_{4l}}+\textbf{i}\partial_{x_{4l+1}} & -\partial_{x_{4l+2}}-\textbf{i}\partial_{x_{4l+3}} \\
                                        \partial_{x_{4l+2}}-\textbf{i}\partial_{x_{4l+3}} & \partial_{x_{4l}}-\textbf{i}\partial_{x_{4l+1}} \\
                                      \end{array}
                                    \right).
\end{equation*}
$z^{k\beta}$'s can be viewed as independent variables and $\nabla_{j\alpha}$'s are derivatives with respect to these variables. The operators $\nabla_{j\alpha}$'s play very important roles in the investigating of regular functions in several quaternionic variables \cite{kang,Wang,wang1,Wang2,Wang3}.
\par

Let $\wedge^{2k}\mathbb{C}^{2n}$ be the complex exterior algebra generated by $\mathbb{C}^{2n}$, $0\leq k\leq n$. Fix a basis
$\{\omega^0,\omega^1,\ldots$, $\omega^{2n-1}\}$ of $\mathbb{C}^{2n}$. Let $\Omega$ be a domain in $\mathbb{R}^{4n}$. Define $d_0,d_1:C_0^\infty(\Omega,\wedge^{p}\mathbb{C}^{2n})\rightarrow C_0^\infty(\Omega,\wedge^{p+1}\mathbb{C}^{2n})$ by \begin{equation*}\begin{aligned}&d_0F=\sum_{k,I}\nabla_{k0 }f_{I}~\omega^k\wedge\omega^I,\\
&d_1F=\sum_{k,I}\nabla_{k1 }f_{I}~\omega^k\wedge\omega^I,\\
&\triangle
F=d_0d_1F,
\end{aligned}\end{equation*}for $F=\sum_{I}f_{I}\omega^I\in C_0^\infty(\Omega,\wedge^{p}\mathbb{C}^{2n})$,  where the multi-index
$I=(i_1,\ldots,i_{p})$ and
$\omega^I:=\omega^{i_1}\wedge\ldots\wedge\omega^{i_{p}}$. Although $d_0,d_1$ are not exterior differential, their behavior is similar to the exterior differential: $d_0d_1=-d_1d_0$;  $d_0^2=d_1^2=0$, $d_0\triangle=d_1\triangle=0$; for $F\in C_0^\infty(\Omega,\wedge^{p}\mathbb{C}^{2n})$, $G\in C_0^\infty(\Omega,\wedge^{q}\mathbb{C}^{2n})$, we have\begin{equation}\label{2.228}d_\alpha(F\wedge G)=d_\alpha F\wedge G+(-1)^{p}F\wedge d_\alpha G,\qquad \alpha=0,1.\end{equation}
Moreover $\triangle u_1\wedge\ldots\wedge\triangle u_n$ satisfies the following remarkable identities:\begin{equation}\label{2.37}\begin{aligned}&\triangle u_1\wedge \triangle
u_2\wedge\ldots\wedge\triangle u_n=d_0(d_1u_1\wedge \triangle
u_2\wedge\ldots\wedge\triangle u_n)\\=&-d_1(d_0u_1\wedge \triangle
u_2\wedge\ldots\wedge\triangle u_n)=d_0d_1(u_1\triangle
u_2\wedge\ldots\wedge\triangle u_n)\\=&\triangle (u_1
\triangle u_2\wedge\ldots\wedge\triangle u_n).
\end{aligned}\end{equation}
To write down the explicit expression, we
define for a function $u\in C^2$,\begin{equation*}\label{2.10}\triangle_{ij}u:=\frac{1}{2}(\nabla_{i0 }\nabla_{j1 }u-\nabla_{i1 }\nabla_{j0 }u),~~~\triangle u=\sum_{i,j=0}^{2n-1}\triangle_{ij}u~\omega^i\wedge
\omega^j.\end{equation*}
\begin{equation*}\label{2.11}\begin{aligned}\triangle
u_1\wedge\ldots\wedge\triangle
u_n&=\sum_{i_1,j_1,\ldots}\triangle_{i_1j_1}u_1\ldots\triangle_{i_nj_n}u_n~\omega^{i_1}\wedge
\omega^{j_1}\wedge\ldots\wedge \omega^{i_n}\wedge
\omega^{j_n}\\&=\sum_{i_1,j_1,\ldots}\delta^{i_1j_1\ldots
i_nj_n}_{01\ldots(2n-1)}\triangle_{i_1j_1}u_1\ldots\triangle_{i_nj_n}u_n~\Omega_{2n},
\end{aligned}\end{equation*}for $u_1,\ldots,u_n\in
C^2$, where $\Omega_{2n}$ is defined as\begin{equation*}\label{2.21}\Omega_{2n}:=\omega^0\wedge
\omega^1\wedge\ldots\wedge\omega^{2n-2}\wedge
\omega^{2n-1},\end{equation*}and $\delta^{i_1j_1\ldots
i_nj_n}_{01\ldots(2n-1)}:=$ the sign of the permutation from $(i_1,j_1,\ldots
i_n,j_n)$ to $(0,1,\ldots,2n-1),$ if $\{i_1,j_1,\ldots,
i_n,j_n\}=\{0,1,\ldots,2n-1\}$; otherwise, $\delta^{i_1j_1\ldots
i_nj_n}_{01\ldots(2n-1)}=0$.

The quaternionic Monge-Amp\`{e}re operator $\text{det}(u)$ given by (\ref{det}) coincides with $\triangle_nu$ while the mixed
Monge-Amp\`{e}re operator ~$\text{det}(u_1,\ldots,u_n)$ ~coincides with
$\triangle_n(u_1,\ldots,
u_n)$. For $C^2$ functions we proved in \cite{wan3} that $$\triangle u_1\wedge\ldots\wedge\triangle u_n=n!\text{det}(u_1,\ldots,u_n)\Omega_{2n}~~\text{and}~~(\triangle u)^n=n!\text{det}(u)\Omega_{2n}.$$

We defined in \cite{wan3} the notions of quaternionic closed positive forms and closed positive currents.  Please see \cite{wan3,wan4,wan7} for detailed definitions. 
\begin{lem}\label{p3.9}$($Stokes-type formula, Lemma 2.1 in \cite{wan4}$)$ Assume that $T$ is a smooth $(2n-1)$-form in $\Omega$, and $h$ is a smooth function with $h=0$ on $\partial\Omega$. Then we have \begin{equation*}\int_\Omega hd_\alpha T=-\int_\Omega d_\alpha h\wedge T,\qquad \alpha=0,1.\end{equation*}
\end{lem}

Bedford-Taylor theory \cite{bed} in complex analysis can be generalized to the quaternionic case. Let $u$ be a locally bounded $PSH$ function and let $T$ be a closed
positive $2k$-current. Then $\triangle u\wedge T$ defined by \begin{equation}\label{3.110}\triangle u\wedge
T:=\triangle(uT),\end{equation}i.e., $(\triangle u\wedge
T)(\eta):=uT(\triangle\eta)$, for any test form $\eta$, is also a closed positive
current. Inductively,
\begin{equation}\label{3.111}\triangle
u_1\wedge\ldots\wedge\triangle u_p\wedge T:=\triangle(u_1\triangle
u_2\wedge\ldots\wedge\triangle u_p\wedge T)\end{equation} is a closed
positive current, when $u_1,\ldots,u_p\in PSH\cap
L_{loc}^\infty(\Omega)$. 

In particular, for $u_1,\ldots,u_n\in PSH\cap
L_{loc}^{\infty} (\Omega)$, $\triangle u_1\wedge\ldots\wedge\triangle u_n=\mu\Omega_{2n}$ for a well defined positive Radon
measure $\mu$.
\begin{lem}\label{t3.2}(1) (Lemma 2.2 in \cite{wan4}) \\Let
$v^1,\ldots,v^k\in PSH\cap L_{loc}^{\infty}(\Omega)$. Let
$\{v_j^1\}_{j\in\mathbb{N}},\ldots,\{v_j^k\}_{j\in\mathbb{N}}$ be
decreasing sequences of PSH functions in $\Omega$ such
that $\lim_{j\rightarrow\infty}v_j^t=v^t$   pointwise in $\Omega$
for each $t $. Then the currents $\triangle
v_j^1\wedge\ldots\wedge\triangle v_j^k $ converge weakly to
$\triangle v^1\wedge\ldots\wedge\triangle v^k $ as
$j\rightarrow\infty$.\\
(2) (Proposition 4.1 in \cite{wan4})\\ Let $\{u_j\}_{j\in \mathbb{N}}$ be a sequence in $PSH\cap L_{loc}^\infty(\Omega)$ that increases to $u\in PSH\cap L_{loc}^\infty(\Omega)$ almost everywhere in $\Omega$ $($with respect to Lebesgue measure$)$. Then the currents $(\triangle u_j)^n$ converge weakly to $(\triangle u)^n$ as $j\rightarrow\infty$.\\
(3) (Proposition 3.2 in \cite{wan4})\\ Let $v^0,\ldots,v^k\in PSH\cap L_{loc}^\infty(\Omega)$. Let
$\{v_j^0\}_{j\in\mathbb{N}},\ldots,\{v_j^k\}_{j\in\mathbb{N}}$ be
decreasing sequences of PSH functions in $\Omega$ such
that $\lim_{j\rightarrow\infty}v_j^t=v^t$ pointwise in $\Omega$
for $t=0,\ldots,k$. Then the currents $v_j^0\triangle
v_j^1\wedge\ldots\wedge\triangle v_j^k $ converge weakly to
$v^0\triangle v^1\wedge\ldots\wedge\triangle v^k $ as
$j\rightarrow\infty$.
\end{lem}

\begin{lem}\label{t2.2}(Comparison principle, Theorem 1.2 in \cite{wan4}) Let $u,v\in PSH\cap L_{loc}^\infty(\Omega)$. If for any $\zeta\in\partial\Omega$,
$$\liminf_{\zeta\leftarrow q\in\Omega}~~(u(q)-v(q))\geq0,$$
then
\begin{equation*}\label{5.1}\int_{\{u<v\}}(\triangle v)^n\leq\int_{\{u<v\}}(\triangle u)^n.
\end{equation*}
\end{lem}

\section{Cegrell's classes $\mathcal{E}_0(\Omega),\mathcal{E}(\Omega),\mathcal{F}(\Omega),\mathcal{E}_p(\Omega),\mathcal{F}_p(\Omega)$}

First we establish a global approximation of negative
$PSH$ functions by decreasing sequences of continuous negative $PSH$
functions with zero boundary value and bounded quaternionic Monge-Amp\`{e}re mass. Cegrell proved in \cite{cegrell2004} this approximation for complex $PSH$ functions on hyperconvex domains. He also proved in \cite{cegrell2} that the exhaustion function of a hyperconvex domain in $\Bbb C^n$ can be chosen in $C^\infty(\Omega)\cap C(\overline{\Omega})$ with bounded mass.
\begin{thm}\label{t1.1}For any function $\varphi\in PSH^-(\Omega)$, there is a decreasing sequence $\{\varphi_j\}$ of PSH functions satisfying the following properties:
$\varphi_{j}$ is continuous on $\bar{\Omega}$ and
$\varphi_{j}|_{\partial\Omega}=0$, $supp~(\triangle\varphi_j)^n\Subset \Omega$ for each $j$, and
$\lim_{j\rightarrow+\infty}\varphi_{j}(q)=\varphi(q)$, for $ q\in\Omega$.
\end{thm}
In order to prove this result, we need the relatively extremal function $   u_{E,\Omega}$ defined in \cite{wan8} as
\beq\label{u*}  u_{E,\Omega}(q)=\sup\{u(q):u\in PSH(\Omega),u\leq0,u|_E\leq -1\},~~~ q\in \Omega,\eeq for a subset $E\Subset\Omega$.
If $E_{1}\subset
E_{2}\subset\Omega_{1}\subset\Omega_{2}$, then
\begin{equation}\label{1.2}u_{E_{1},\Omega_{1}}\geq u_{E_{2},\Omega_{1}}\geq
u_{E_{2},\Omega_{2}}.\end{equation}
Its upper semicontinuous regularization $u_{E,\Omega}^*\in PSH(\Omega)$ (by Proposition \ref{p1.1}) satisfies $-1\leq u_{E,\Omega}^*\leq0$ in $\Omega$ and $(\triangle u_{E,\Omega}^*)^n=0$ on $\Omega\setminus \overline{E}$ for any $E\Subset \Omega$. This function was defined in the same way as the relatively extremal function given by Demailly \cite{Demailly1991} for the complex case (see also \cite{klimek} for the detailed discussion).

 \begin{lem}\label{l1.1}Suppose that $K\subset\Omega$ is
a compact set such that $u_{K,\Omega}^{*}\equiv-1$ on $K$. Then
$u_{K,\Omega}$ is continuous in $\Omega$.
\end{lem}
\proof Denote $u=u_{K,\Omega}$. Let $\rho$ be an exhaustion function of $\Omega$ such that $\rho<-1$ on
$K$, then $\rho\leq u$ in $\Omega$. For each $\varepsilon\in(0,1)$,
there exists $\eta>0$ such that
$u-\frac{\varepsilon}{2}\leq-\frac{\varepsilon}{2}<\rho$ in
$\Omega\setminus\Omega_{\eta}$ and $K\subset\Omega_{\eta}$, where
$\Omega_{\eta}=\{q\in\Omega:dist(q,\partial\Omega)>\eta\}.$ Note that $u^{*}$ is PSH, thus is upper
semi-continuous. Denote by $u_{\delta}$ the standard regularization of
$u$, defined on $\Omega_{\delta}$, then $(u^{*})_{\delta}$
decreasingly converges to $u^{*}$ as $\delta\rightarrow0$ by Proposition \ref{p1.1}. Since
$u^{*}=u$ almost everywhere in $\Omega$, we have $(u^{*})_{\delta}=u_{\delta}$.
By Dini theorem, $u_{\delta}$ converges to $u^{*}$ uniformly on
any compact set. Therefore we can find a uniform $\delta>0$ such
that $u_{\delta}-\varepsilon<u^{*}-\frac{\varepsilon}{2}<\rho$ on
$\partial\Omega_{\eta}$ and $u_{\delta}-\varepsilon<-1$ on $K$.
Define\begin{equation*}g_{\varepsilon}=\begin{cases}\rho,\,\,\,\,\,\,\,\,\,\,\,\,\,\,\,\,\,\,\,\,\,\,\,\,\,\,\,\,\,\,\,\,\,\,\,\,\,\,\,\,\,\,\text{in}\,\,\Omega\setminus\Omega_{\eta},\\
\max\{u_{\delta}-\varepsilon,\rho\},\,\,\,\,\,\,\,\text{in}\,\,\Omega_{\eta}.
\end{cases}
\end{equation*}Then $g_{\varepsilon}\in PSH\cap C(\Omega)$ by
Proposition \ref{p1.1} (4) and $g_{\varepsilon}<-1$ on $K$. Thus
$g_{\varepsilon}\leq u$ in $\Omega$. On the other hand,
$u-\varepsilon\leq\max\{u_{\delta}-\varepsilon,\rho\}\leq
g_{\varepsilon}$ in $\Omega$. It follows that $u$ is
continuous in $\Omega$.
\endproof

\begin{cor}\label{c1.1}Suppose that $K\subset\Omega$
is compact and is the union of a family of closed balls. Then
$u_{K,\Omega}=u_{K,\Omega}^{*}$ is continuous in $\Omega$.
\end{cor}
\begin{proof}By Lemma \ref{l1.1}, it suffices to prove that $u_{K,\Omega}$
is continuous at $\partial K$. Let $b\in\partial K$,
and choose $a\in K$, $R>r>0$ such that $b\in\bar{B}(a,r)\subset
K$ and $B(a,R)\subset\Omega$. Then by (\ref{1.2}) we have
$u_{K,\Omega}(q)\leq u_{\bar{B}(a,r),\Omega}(q)\leq
u_{\bar{B}(a,r),B(a,R)}(q)$ for $q\in B(a,R)$. It follows from the
fact:
\begin{equation*}u_{\bar{B}(a,r),B(a,R)}(q)=\max\left\{-1,\frac{R^{-2}-\|q-a\|^{-2}}{r^{-2}-R^{-2}}\right\}
\end{equation*}that $\lim_{q\rightarrow b}u_{K,\Omega}(q)=-1$.
\end{proof}

\emph{Proof of Theorem \ref{t1.1}}. By Corollary \ref{c1.1}, $u_{\bar{B},\Omega}$ is
continuous in $\Omega$. We have $(\triangle u_{\bar{B},\Omega})^n=0$ in $\Omega\setminus
\bar{B}$ (cf. \cite[Proposition 3.2]{wan8}), thus $\text{supp}~(\triangle u_{\bar{B},\Omega})^n\Subset \Omega$. Let $\rho$ be an exhaustion function of
$\Omega$ such that $C\rho\leq-1$ in $\bar{B}$. Then $C\rho\leq
u_{\bar{B},\Omega}\leq0$. It follows from $\rho(q)\rightarrow0$, $q\rightarrow\partial\Omega$, that $u_{\bar{B},\Omega}(q)\rightarrow0$. Then we can follow the lines in \cite[Theorem 2.1]{cegrell2004} to complete the proof.
\qed

The following lemma implies that continuous functions in the class $\mathcal {E}_0(\Omega)$ can serve as ``test functions". One can prove this result in the same way as the complex $PSH$ case.
\b{lem}\label{l3.2} $C_0^\infty(\Omega)\subset \mathcal {E}_0\cap C(\overline{\Omega})-\mathcal {E}_0\cap C(\overline{\Omega}).$
\e{lem}

In \cite{wan4} we defined for $u,v\in C^2(\Omega)$,
\begin{equation}\label{6.1}\gamma( u,v):=\frac{1}{2}(d_0u\wedge d_1v-d_1u\wedge d_0v).
\end{equation} 
In \cite{wan7} we proved that when all functions are locally bounded, the mixed product $\gamma(
u, u)\wedge\triangle
w^1\wedge\ldots\wedge \triangle
w^k$
is well defined as a positive $(2k+2)$-current. Moreover, the Chern-Levine-Nirenberg type estimate holds for the positive current $\gamma(
u, u)\wedge\triangle
w^1\wedge\ldots\wedge \triangle
w^k$, which is also continuous on decreasing sequences. And we obtained some properties of the well defined current $\gamma(u,v)\wedge\triangle
w^1\wedge\ldots\wedge \triangle w^k.$ These results enable us to integrate by parts and then get some useful estimates. See Proposition 3.1 in \cite{wan7} for the detailed discussion.

\begin{lem}\label{p6.1}(Corollary 3.1 in \cite{wan7})
Let $u,v,w^1,\ldots,w^{n-1}\in PSH\cap L_{loc}^\infty(\Omega)$. Then
\begin{equation*}\label{6.6}\int_\Omega\gamma( u,
v)\wedge T\leq\left(\int_\Omega\gamma(
u, u)\wedge T\right)^\frac{1}{2}\left(\int_\Omega\gamma( v,v)\wedge T\right)^\frac{1}{2},
\end{equation*}where $T=\triangle
w^1\wedge\ldots\wedge \triangle
w^{n-1}$.\end{lem}

\b{pro}\label{p3.1} Suppose that $u,v\in PSH^-\cap L_{loc}^\infty(\Omega)$ and that $T$ is a closed positive $(2n-2)$-current. If $\lim_{q\rightarrow\partial\Omega}u(q)=0$, then $$\int v\triangle u\wedge T\leq \int u\triangle v\wedge T.$$ If in addition $\lim_{q\rightarrow\partial\Omega}v(q)=0$, integration by parts is allowed:
\beq\label{integration by parts}\int v\triangle u\wedge T= \int u\triangle v\wedge T=\int -\gamma(u,v)\wedge T.
\eeq
\e{pro}
\proof First noting that for $\varphi\in C_0^\infty(\Omega)$, by the inductive definition (\ref{3.111}) we have \beq\label{3.2}\int \varphi\triangle v\wedge T=\int v\triangle \varphi\wedge T .\eeq
Hence the first identity of (\ref{integration by parts}) is true with $\varphi\in C_0^\infty(\Omega)$ in place of $u$. Now for $u\in PSH^-\cap L_{loc}^\infty(\Omega)$, we denote $u_\varepsilon=\max\{u,-\varepsilon\}$. Then $u-u_\varepsilon=\min\{0,u+\varepsilon\}$ is a compactly supported function decreasing uniformly to $u$ as $\varepsilon\rightarrow0$. So $$\int u\triangle v\wedge T=\lim_{\varepsilon\rightarrow0}\int (u-u_\varepsilon)\triangle v\wedge T.$$
And by (\ref{3.2}) we have $$\int (u-u_\varepsilon)\triangle v\wedge T=\lim_{\delta\rightarrow0}\int (u-u_\varepsilon)_\delta\triangle v\wedge T=\lim_{\delta\rightarrow0}\int v \triangle (u-u_\varepsilon)_\delta\wedge T,$$ where $(u-u_\varepsilon)_\delta$ denotes the standard regularization of $u-u_\varepsilon$ such that $(u-u_\varepsilon)_\delta$ decreases to $u-u_\varepsilon$ as $\delta$ tends to $0$. Fix an open set $\Omega'\Subset\Omega$ such that the set $\{u<-\varepsilon\}\Subset \Omega'$, then the support of $\triangle (u-u_\varepsilon)_\delta$ is in $\Omega'$ for $\delta$ small enough. Note that $(u_\varepsilon)_\delta$ is also plurisubharmonic by Proposition \ref{p1.1}, thus $\triangle (u_\varepsilon)_\delta\wedge T\geq0$. It follows from Lemma \ref{t3.2} that \ben\b{aligned}\int (u-u_\varepsilon)\triangle v\wedge T=&\lim_{\delta\rightarrow0}\int_{\Omega'} v \triangle (u-u_\varepsilon)_\delta\wedge T\geq \varlimsup_{\delta\rightarrow0}\int_{\Omega'} v \triangle u_\delta\wedge T\\=&\int_{\Omega'} v \triangle u\wedge T\rightarrow \int_{\Omega} v \triangle u\wedge T~~~~\text{as}~\Omega'\rightarrow\Omega.\e{aligned}\een
Let $\varepsilon\rightarrow0$ to get $\int v\triangle u\wedge T\leq \int u\triangle v\wedge T.$

To show (\ref{integration by parts}), it suffices to prove the second identity of (\ref{integration by parts}) for the smooth case and repeat the above approximation process for the general case. Since $T$ is closed, by (\ref{2.228}) (\ref{2.37}) and Stokes-type formula (Lemma \ref{p3.9}) we have
\ben \b{aligned}\int v \triangle u\wedge T=& \int u\triangle v\wedge T=\frac{1}{2}\int u(d_0d_1-d_1d_0) v\wedge T\\=&\frac{1}{2}\int ud_0(d_1 v\wedge T)-\frac{1}{2}\int ud_1(d_0 v\wedge T)\\=&-\frac{1}{2}\int d_0u\wedge d_1 v\wedge T+\frac{1}{2}\int d_1u\wedge d_0 v\wedge T=-\int\gamma(u,v)\wedge T.
\e{aligned}\een
\endproof

\b{pro}\label{l3.1} Suppose that $h,u_1,u_2,v_1,\ldots,v_{n-p-q}\in \mathcal {E}_0(\Omega)$, $1\leq p,q<n$. Let $T=\triangle v_1\wedge\ldots\wedge\triangle v_{n-p-q}$. Then
\ben \label{3.1}\b{aligned}&\int -h(\triangle u_1)^p\wedge (\triangle u_2)^q\wedge T\\
\leq &\left[\int -h(\triangle u_1)^{p+q}\wedge T\right]^{\frac{p}{p+q}}\left[\int -h(\triangle u_2)^{p+q}\wedge T\right]^{\frac{q}{p+q}}.
\e{aligned}\een
\e{pro}
\proof We first prove this conclusion for the case $p=q=1$. By Proposition \ref{p3.1} and Lemma \ref{p6.1} we have\ben\b{aligned}&\int -h\triangle u_1\wedge \triangle u_2\wedge T=\int -u_1\triangle u_2\wedge \triangle h\wedge T=\int \gamma(u_1, u_2)\wedge \triangle h\wedge T\\&\qquad\leq \left[\int \gamma(u_1, u_1)\wedge \triangle h\wedge T\right]^{1/2}\left[\int \gamma(u_2, u_2)\wedge \triangle h\wedge T\right]^{1/2}\\&\qquad=\left[\int -u_1\triangle u_1\wedge \triangle h\wedge T\right]^{1/2}\left[\int -u_2\triangle u_2\wedge \triangle h\wedge T\right]^{1/2}\\&\qquad=\left[\int -h(\triangle u_1)^2\wedge T\right]^{1/2}\left[\int -h(\triangle u_2)^2\wedge T\right]^{1/2}.
\e{aligned}\een  Then one can complete the proof by following the inductive process in \cite[Lemma 5.4]{cegrell2004}.
\endproof

\b{cor}\label{c3.1} Suppose that $h,u_1,\ldots,u_n\in \mathcal {E}_0(\Omega)$. Then
\ben \int -h\triangle u_1\wedge \ldots\wedge \triangle u_n
\leq \left[\int -h(\triangle u_1)^n\right]^{\frac{1}{n}} \ldots\left[\int -h(\triangle u_n)^n\right]^{\frac{1}{n}}.
\een\end{cor}

With the necessary ingredients (Theorem \ref{t1.1}, Lemma \ref{l3.2} and Proposition \ref{p3.1}), we are to show that the quaternionic Monge-Amp\`ere operator is well defined for functions in $\mathcal {E}(\Omega)$ and $\mathcal {E}_p(\Omega)$, and that integration by parts is allowed in $\mathcal {F}(\Omega)$ and $\mathcal {E}_p(\Omega)$, by following the ideas of Cegrell \cite{cegrell1998,cegrell2004}.

\b{thm}\label{t4.1} Suppose that $u^k\in \mathcal {E}(\Omega)$ and $\{g^k_{j}\}_{j=1}^\infty\subset \mathcal {E}_0(\Omega)$ decreases to $u^k$ as $j\rightarrow+\infty$, $k=1,\ldots,n$. Then the sequence of measures $\triangle g^1_{j}\wedge\ldots \wedge \triangle g^n_{j}$ converges weakly to a positive measure which does not depend on the choice of the sequences $\{g^k_{j}\}_{j=1}^\infty$. We then define $\triangle u^1\wedge \ldots \wedge \triangle u^n$ to be this weak limit.
\e{thm}
\proof See Theorem 4.2 in \cite{cegrell2004}.\endproof

\b{pro}\label{p3.3} Suppose that $u^k\in \mathcal {F}(\Omega)$, $k=1,\ldots,n$, and $\{g^k_{j}\}_{j=1}^\infty\subset\mathcal {E}_0(\Omega)$ decreases to $u^k$, $j\rightarrow+\infty$, such that $$\sup_{j,k}\int_\Omega(\triangle g^k_{j})^n<+\infty.$$ Then for each $h\in PSH^-(\Omega)$ we have
\ben \lim_{j\rightarrow +\infty}\int_\Omega h \triangle g^1_{j}\wedge\ldots \wedge \triangle g^n_{j}=\int_\Omega h \triangle u^1\wedge \ldots \wedge \triangle u^n.
\een
\e{pro}
\proof Assume first that $h\in  \mathcal {E}_0\cap C(\Omega)$. Fix $\varepsilon>0$ small enough and consider $h_\varepsilon:=\max\{h,-\varepsilon\}$. The function $h-h_\varepsilon$ is continuous and compactly supported in $\Omega$. By Theorem \ref{t4.1} we have
$$ \lim_{j\rightarrow +\infty}\int_\Omega (h-h_\varepsilon) \triangle g^1_{j}\wedge\ldots \wedge \triangle g^n_{j}=\int_\Omega (h-h_\varepsilon) \triangle u^1\wedge \ldots \wedge \triangle u^n. $$ Noting that $|h_\varepsilon|\leq \varepsilon$ and $\sup_j\int  \triangle g^1_{j}\wedge\ldots \wedge \triangle g^n_{j}<+\infty,$ we get the result for $h\in  \mathcal {E}_0\cap C(\Omega)$.

Now suppose that $h\in PSH^-(\Omega)$ and that $\int_\Omega h \triangle u^1\wedge \ldots \wedge \triangle u^n$ is finite. By Theorem \ref{t1.1} we can choose $h_j\in \mathcal {E}_0\cap C(\Omega)$ decreasing to $h$ and get the result. If $\int_\Omega h \triangle u^1\wedge \ldots \wedge \triangle u^n=-\infty$, then $\lim_{j\rightarrow +\infty}\int_\Omega h \triangle g^1_{j}\wedge\ldots \wedge \triangle g^n_{j}=-\infty$.
\endproof

\b{cor} \label{c4.1} Suppose that $\{g_{j}\}_{j=1}^\infty\subset\mathcal {E}_0(\Omega)$ decreases to $u\in \mathcal {F}(\Omega)$, $j\rightarrow+\infty$, such that $$\sup_{j}\int_\Omega(\triangle g_{j})^n<+\infty.$$ Then for each $h\in \mathcal {E}_0(\Omega)$ we have the weak convergence
$ h (\triangle g_{j})^n\rightarrow h (\triangle u)^n. $
\e{cor}

\b{rem}\label{r4.1}
It follows from the definition of $\mathcal{F}$ and Proposition \ref{p3.3} that Corollary \ref{c3.1} also holds for $u_1,\ldots,u_n\in \mathcal {F}(\Omega)$. In particular, for $u_1,\ldots,u_n\in \mathcal {F}(\Omega)$, we have\beq \label{4.111}\int \triangle u_1\wedge \ldots\wedge \triangle u_n
\leq \left(\int (\triangle u_1)^n\right)^{\frac{1}{n}} \ldots\left(\int (\triangle u_n)^n\right)^{\frac{1}{n}}.
\eeq \e{rem}

The functions in $\mathcal{F}(\Omega)$ have boundary value zero in some sense, which can been seen in the following integration by parts formula.

\b{thm}\label{integrationthm} Let $u,v,w^1,\ldots,w^{n-1}\in \mathcal{F}(\Omega)$ and $T=\triangle w^1\wedge\ldots\wedge \triangle w^{n-1}$. Then
\ben \int u\triangle v\wedge T=\int v\triangle u\wedge T.\een
\e{thm}
\proof Let $u_j,v_j,w^1_j,\ldots,w^{n-1}_j$ be sequences in $\mathcal{E}_0(\Omega)\cap C(\Omega)$ decreasing to $u,v,w^1,$  $\ldots,w^{n-1}$ respectively such that their total masses are uniformly bounded:
\ben \sup_j \int \triangle v_j\wedge T_j<\infty,~~~\sup_j\int \triangle u_j\wedge T_j<\infty,~~~\text{where}~T_j=\triangle w^1_j\wedge \ldots \wedge \triangle w^{n-1}_j.\een
Theorem \ref{t4.1} and Proposition \ref{p3.1} imply that $\int h\triangle u_j\wedge T_j$ decreases to $ \int h\triangle u\wedge T$. Fix $k$, for any $j>k$ we have
$$\int v_k\triangle u_{k}\wedge T_k\geq \int v_k\triangle u_j\wedge T_j\geq \int v_j\triangle u_j\wedge T_j.$$
Thus the sequence $\int v_j\triangle u_j\wedge T_j$ decreases to some $a\in[-\infty,+\infty)$. It follows from Proposition \ref{p3.3} that $\int v_k\triangle u\wedge T\geq a$, thus $\int v\triangle u\wedge T\geq a$. On the other hand, we have $$\int v\triangle u\wedge T\leq \int v_k\triangle u\wedge T=\lim_j\int v_k\triangle u_j\wedge T_j\leq \int v_k\triangle u_k\wedge T_k.$$ It follows that $\int v\triangle u\wedge T=a$.
\endproof

As in the complex case, we define the \textit{quaternionic $p$-energy} of $\varphi\in\mathcal{E}_0(\Omega)$ to be
\ben\label{energy} E_p(\varphi):=\int_\Omega(-\varphi)^p(\triangle\varphi)^n
\een
and the \textit{mutual quaternionic $p$-energy} of $\varphi_0,\ldots,\varphi_n\in\mathcal{E}_0(\Omega)$ to be
\ben\label{m-energy} E_p(\varphi_0,\varphi_1,\ldots,\varphi_n):=\int_\Omega(-\varphi_0)^p\triangle\varphi_1\wedge\ldots\wedge\triangle\varphi_n,~~p\geq1.\een If $p=1$ we denote by $E(\varphi)=E_1(\varphi)$, and we say that $E(\varphi)$ is the \textit{quaternionic energy} of $\varphi\in \mathcal{E}_0(\Omega)$.
We get the following useful H\"older-type inequality by following the method of Persson \cite{persson} for the complex case. This energy estimate plays an important role in the variational approach in next section.
\b{thm}\label{energy estimate}(Energy estimate) Let $u,v^1,\ldots,v^n\in\mathcal{E}_0(\Omega)$ and $p\geq1$. We have
\ben E_p(u,v^1,\ldots,v^n)\leq D_p E_p(u)^{\frac{p}{n+p}}E_p(v^1)^{\frac{1}{n+p}}\ldots E_p(v^n)^{\frac{1}{n+p}},
\een where $D_1=1$, $D_p=p^{p\alpha(n,p)/p-1}$ for $p>1$ and $$\alpha(n,p)=(p+2)(\frac{p+1}{p})^{n-1}-(p+1).$$
\e{thm}

\proof By Theorem 4.1 in \cite{persson}, it is sufficient to prove that for $u,v,v^1,\ldots,v^{n-1}\in\mathcal{E}_0(\Omega)$,
\beq\label{holder}E_p(u,v,v^1,\ldots,v^{n-1})\leq C_p E_p(u,u,v^1,\ldots,v^{n-1})^{\frac{p}{p+1}}E_p(v,v,v^1,\ldots,v^{n-1})^{\frac{1}{p+1}},
\eeq where $C_1=1$ and $C_p=p^{\frac{p}{p-1}}$ for $p>1$. Let $T=
\triangle v^1\wedge\ldots\wedge\triangle v^{n-1}$. By Proposition \ref{p3.1} and Lemma \ref{p6.1} we have
\ben \b{aligned}\int (-u)\triangle v\wedge T=\int \gamma(u,v)\wedge T&\leq \left(\int \gamma(u,u)\wedge T\right)^{\frac{1}{2}}\left(\int \gamma(v,v)\wedge T\right)^{\frac{1}{2}}\\
&=\left(\int (-u)\triangle u\wedge T\right)^{\frac{1}{2}}\left(\int (-v)\triangle v\wedge T\right)^{\frac{1}{2}}.
\e{aligned}\een
This implies (\ref{holder}) in the case of $p=1$.

For $p>1$, by (\ref{2.228}) and Proposition \ref{p3.1} we can integrate by part to get
\ben \b{aligned}\int& (-u)^p\triangle v\wedge T=\int -\gamma((-u)^p,v)\wedge T=p\int (-u)^{p-1}\gamma(u,v)\wedge T\\ =&p\int (-v)\gamma((-u)^{p-1},u)\wedge T+p\int (-v)(-u)^{p-1}\triangle u\wedge T\\=&-p(p-1)\int (-v)(-u)^{p-2}\gamma(u,u)\wedge T+p\int (-v)(-u)^{p-1}\triangle u\wedge T\\\leq&p\int (-v)(-u)^{p-1}\triangle u\wedge T,
\e{aligned}\een where the last inequality follows from the fact that $\gamma(u,u)\wedge T$ is positive by Proposition 3.1 in \cite{wan7}. Therefore it follows from H\"older inequality that
\beq\label{4.51} \int (-u)^p\triangle v\wedge T\leq p\left[\int (-u)^p\triangle u\wedge T\right]^{\frac{p-1}{p}}\left[\int (-v)^p\triangle u\wedge T\right]^{\frac{1}{p}}.\eeq
By changing the roles of $u$ and $v$ in (\ref{4.51}), we get
\beq\label{4.52} \int (-v)^p\triangle u\wedge T\leq p\left[\int (-v)^p\triangle v\wedge T\right]^{\frac{p-1}{p}}\left[\int (-u)^p\triangle v\wedge T\right]^{\frac{1}{p}}.\eeq Combine (\ref{4.51}) with (\ref{4.52}) to get (\ref{holder}).
\endproof

From the above results, we can get the convexity of the classes $\mathcal{E}_0(\Omega),\mathcal{E}(\Omega),$ $\mathcal{F}(\Omega),\mathcal{E}_p(\Omega),\mathcal{F}_p(\Omega)$, in the same way as in \cite{cegrell1998,cegrell2004}. Moreover, the following results can be obtained by repeating the arguments in  \cite{cegrell1998,cegrell2004}.

\b{defi}\label{defi1}Denote by $\mathcal{K}\subset PSH^-(\Omega)$ such that:\\
(a). If $u\in \mathcal{K}$, $v\in PSH^-(\Omega)$, then $\max\{u,v\}\in \mathcal{K}$. \\
(b). If $u\in \mathcal{K}$, $\varphi_j\in PSH^-\cap L_{loc}^\infty(\Omega)$, $\varphi_j \searrow u$, $j\rightarrow +\infty$, then $(\triangle \varphi_j)^n$ is weakly convergent.
\e{defi}

\b{cor}\label{c4.3}(1) The classes $\mathcal{E}_0(\Omega),\mathcal{E}(\Omega),\mathcal{F}(\Omega),\mathcal{E}_p(\Omega),\mathcal{F}_p(\Omega)$, $p\geq1$, are convex and have property (a) of Definition \ref{defi1}. \\
(2) The class $\mathcal{E}(\Omega)$ has property (a) and (b) of Definition \ref{defi1}. If $\mathcal{K}$ has property (a) and (b), then $\mathcal{K}\subset \mathcal{E}.$ This means that $\mathcal{E}$ is the largest class satisfying property (a) and (b).
\e{cor}

The following result follows from the energy estimate theorem (Theorem \ref{energy estimate}) directly.
\b{cor}\label{estimatecor}Let $u,v\in \mathcal{E}_0(\Omega)$ and $u\leq v$. Then $E_p(v)\leq CE_p(u)$, where the constant $C$ is independent of $u,v$. In particular for $p=1$ we have $E(v)\leq E(u)$.
\e{cor}
As an application of the energy estimate theorem, we can get the following estimate of quaternionic Monge-Amp\`{e}re measure in terms of quaternionic $C_n$-capacity (introduced in \cite{wan7,wan4,wan8}) and the quaternionic $p$-energy. Lu \cite{lu} proved this result for the complex Hessian measure and $m$-subharmonic functions.
\b{cor} Let $U$ be an open subset of $\Omega$ and $\varphi\in  \mathcal{E}_0(\Omega)$, $p\geq1$. Then \ben \int_U(\triangle \varphi)^n\leq M C_n(U)^{\frac{p}{p+n}} E_p(\varphi)^{\frac{n}{p+n}},\een where $M$ is a constant depends only on $p$ and $n$.
\e{cor}
\proof We can assume that $U$ is relatively compact in $\Omega$. Denote by $u=u_U^*$ the regularized relatively extremal function of $U$ in $\Omega$. Then $u\in \mathcal{E}_0(\Omega)$. By Theorem \ref{energy estimate} and the definition of $C_n$-capacity we have
\ben \b{aligned} \int_U(\triangle \varphi)^n&\leq \int_\Omega (-u)^p(\triangle \varphi)^n\leq D_p E_p(u)^{\frac{p}{p+n}}E_p(\varphi)^{\frac{n}{p+n}}\\ &\leq  D_p \left(\int_{\Omega}(\triangle u)^n\right)^{\frac{p}{p+n}}E_p(\varphi)^{\frac{n}{p+n}}\leq
D_p C_n(U)^{\frac{p}{p+n}} E_p(\varphi)^{\frac{n}{p+n}}.
\e{aligned}\een\endproof

\b{thm}\label{t3.15} Suppose that $u^k\in \mathcal{E}_p(\Omega)$, $k=1,\ldots,n, ~p\geq1$, and $\{g_j^k\}_{j=1}^\infty\subset\mathcal{E}_0(\Omega)$ decreases to $u^k$, $j\rightarrow+\infty$, such that $$\sup_{j,k}E_p(g_j^k)<+\infty.$$ Then the sequence of measures $\triangle g_j^1\wedge\ldots\wedge\triangle g_j^n$ is weakly convergent to a positive measure and the limit is independent of the particular sequence. We then define $\triangle u^1\wedge\ldots\wedge\triangle u^n$ to be this weak limit.
\e{thm}
\proof Suppose first that $\sup_{j,k}(\triangle g_j^k)^n<+\infty.$ Then the conclusion follows from the proof of Theorem 4.2 in \cite{cegrell2004}. Hence it suffices to remove the restriction $\sup_{j,k}(\triangle g_j^k)^n<+\infty.$ Without loss of generality we can assume that all $g_j^k$ are continuous on $\Omega$. Let $K$ be a given compact subset of $\Omega$. Consider the function $$h_j^k:=\sup\left\{v\in PSH^- (\Omega): v\leq g_j^k ~\text{on}~K\right\}.$$ Then $h_j^k=h_j^{k*}\in PSH\cap L^\infty(\Omega)$ by Proposition \ref{p1.1}, $h_j^k$ decreases on $\Omega$ and $h_j^k\searrow u^k$ on $K$.
And $supp(\triangle h_j^k)^n\subset K$ (since $h_j^k$ is maximal in $\Omega\setminus K$ cf. \cite{wan4}). Note that $g_j^k\leq h_j^k$. It follows from Corollary \ref{estimatecor} that $\sup_{j,k}E_p(h_j^k)<+\infty$. Thus $h_j^k$ decreases to some function $v^k\in \mathcal{E}_p(\Omega)$. And $u^k=v^k$ on $K$. The theorem is thus proved.
\endproof
With the same proof as for Theorem \ref{t3.15}, we can also prove the following convergence result.
\b{cor} Suppose that $u^k\in \mathcal{E}_p(\Omega)$, $k=1,\ldots,n, ~p\geq1$, and $\{g_j^k\}_{j=1}^\infty\subset\mathcal{E}_p(\Omega)$ decreases to $u^k$, $j\rightarrow+\infty$. Then the sequence of measures $\triangle g_j^1\wedge\ldots\wedge\triangle g_j^n\longrightarrow\triangle u^1\wedge\ldots\wedge\triangle u^n$ weakly.
\e{cor}

Integration by parts is also allowed in $\mathcal{E}_p(\Omega)$, $p\geq1$, as in Theorem \ref{integrationthm} for the class $\mathcal{F}(\Omega)$.
\b{thm}\label{t4.4}Let $u,v,w^1,\ldots,w^{n-1}\in \mathcal{E}_p(\Omega)$ and $T=\triangle w^1\wedge\ldots \wedge \triangle w^{n-1}$. Then $$\int u\triangle v\wedge T=\int v\triangle u\wedge T.$$
\e{thm}
\proof As in the proof of Proposition \ref{p3.3}, by Theorem \ref{t3.15} we have for $h\in \mathcal{E}_0\cap C(\Omega)$ and $u^k\in\mathcal{E}_p(\Omega)$, $k=1,\ldots,n$,
$$\lim_{j\rightarrow+\infty}\int h\triangle g_j^1\wedge \ldots\wedge \triangle g_j^n=\int h\triangle u^1\wedge \ldots\wedge \triangle u^n,$$ where $\{g_j^k\}$ is the sequence as in Theorem \ref{t3.15}. The conclusion follows by repeating the same arguments as in the proof of Theorem \ref{integrationthm}.
\endproof

\b{pro}\label{t4.6} Let $u,v\in \mathcal{E}_p(\Omega)$ (or $\mathcal{F}(\Omega)$) such that $u\leq v$ on $\Omega$. Then $$\int_\Omega (\triangle u)^n\geq \int_\Omega (\triangle v)^n.$$
\e{pro}
\proof Let $h \in\mathcal{E}_0\cap C(\Omega)$. And let $\{u_j\},\{v_j\}$ be sequences in $\mathcal{E}_0(\Omega)$ decreasing to $u,v$ respectively as in the definition of $\mathcal{E}_p(\Omega)$ (or $\mathcal{F}(\Omega)$). We can assume that $u_j\leq v_j$ on $\Omega$, so we can integrate by part to get $$\int_\Omega (-h)(\triangle u_j)^n\geq \int_\Omega (-h)(\triangle v_j)^n.$$ For $u,v\in \mathcal{E}_p(\Omega)$ (or $\mathcal{F}(\Omega)$), by Theorem \ref{t3.15} (or Proposition \ref{p3.3}) we can let $j\rightarrow+\infty$ to get $$\int_\Omega (-h)(\triangle u)^n\geq \int_\Omega (-h)(\triangle v)^n.$$
To complete the proof, it suffices to let $h$ decrease to $-1$.
\endproof
The following results can be obtained by repeating the arguments in  \cite{cegrell1998,cegrell2004}.
\b{pro}\label{p4.33} (1) If $u_j\in \mathcal{E}_p$, $u_j$ converges increasingly to $u$ as $j\rightarrow+\infty$, then $u\in \mathcal{E}_p$ and $(\triangle u_j)^n\rightarrow (\triangle u)^n$ as $j\rightarrow+\infty$. \\
(2) If $u\in \mathcal{E}_1(\Omega)$, then $\int (-u)(\triangle u)^n<+\infty$. If $\{u_j\}$ is a sequence in $\mathcal{E}_0(\Omega)$ decreasing to $u$, then $\int u_j(\triangle u_j)^n\searrow\int u(\triangle u)^n$, i.e., $E(u_j)\nearrow E(u)$.\\
(3) Suppose that $S$ is a Q-polar subset of $\Omega$. Then there is a $\psi\in \mathcal{E}_1(\Omega)$ such that $S\subset \{\psi=-\infty\}$.
\e{pro}

In \cite{wan8} we proved a Demailly-type inequality of quaternionic Monge-Amp\`ere measures for functions in $PSH\cap L_{loc}^\infty(\Omega)$. Now we can extend this inequality to the classes $\mathcal{F}(\Omega),\mathcal{E}_p(\Omega)$.

\b{lem}\label{p2}(Proposition 3.5 in \cite{wan8}) Let $u,v$ be locally bounded PSH functions on $\Omega$. Then \ben (\triangle\max\{u,v\})^n\geq \chi_{\{u\geq v\}} (\triangle u)^n+\chi_{\{u< v\}} (\triangle v)^n.
\een
\e{lem}

\b{thm}\label{t4.7}Let $u,u^1,\ldots,u^{n-1}\in\mathcal{F}(\Omega)$ (or $\mathcal{E}_p(\Omega)$), $v\in PSH^-(\Omega)$ and $T=\triangle u^1\wedge \ldots\wedge \triangle u^{n-1}$. Then $$\triangle \max(u,v)\wedge T\big|_{\{u>v\}}=\triangle u\wedge T\big|_{\{u>v\}}.$$
\e{thm}
\proof First assume that $v\equiv a$ with some constant $a<0$. By Theorem \ref{t1.1} there exist $u_j,u_j^k\in \mathcal{E}_0(\Omega)\cap C(\overline{\Omega})$ such that $u_j$ decreases to $u$ and $u_j^k$ decreases to $u^k$ for each $ k=1, \ldots,n-1$. Since $\{u_j>a\}$ is open, one has $$\triangle \max(u_j,a)\wedge T_j\big|_{\{u_j>a\}}=\triangle u_j\wedge T_j\big|_{\{u_j>a\}},$$ where $T_j=\triangle u^1_j\wedge \ldots\wedge \triangle u^{n-1}_j$. As $\{u>a\}\subset\{u_j>a\}$ we obtain $$\triangle \max(u_j,a)\wedge T_j\big|_{\{u>a\}}=\triangle u_j\wedge T_j\big|_{\{u>a\}}.$$ It follows from Proposition \ref{p3.3} and Theorem \ref{t3.15} that as $j\rightarrow+\infty$
\ben \b{aligned}\max(u-a,0)\triangle( \max(u_j,a))\wedge T_j&\longrightarrow \max(u-a,0)\triangle( \max(u,a))\wedge T\\ \max(u-a,0)\triangle u_j\wedge T_j&\longrightarrow \max(u-a,0)\triangle u\wedge T.
\e{aligned}\een Hence $$\max(u-a,0)[\triangle( \max(u,a))-\triangle u]\wedge T=0.$$
It follows that $\triangle( \max(u,a))\wedge T=\triangle u\wedge T$ on the set $\{u>a\}$.

Now in the general case for $v\in PSH^-(\Omega)$. Since $\{u>a\}=\cup_{a\in \mathbb{Q}^-}\{u>a>v\}$, it suffices to repeat the arguments in Theorem 4.1 in \cite{hiep}.
\endproof
\b{cor}\label{l3.19'} For $u,v\in \mathcal{E}_p(\Omega)$, $p\geq 1$, 
$$\chi_{\{u>v\}}(\triangle u)^n=\chi_{\{u>v\}}(\triangle \max\{u,v\})^n.$$
\e{cor}
\section{The variational approach}
\subsection{The energy functional}
For a positive Radon measure $\mu$ in $\Omega$, the \emph{energy functional} $\mathcal{F}_\mu:\mathcal{E}_1(\Omega)\rightarrow \mathbb{R}$ is defined by $$\mathcal{F}_\mu(u)=\frac{1}{n+1}E(u)+L_\mu(u),$$ where $L_\mu(u)=\int u d\mu$. Recall that the quaternionic energy $E(\varphi)$ of $\varphi\in \mathcal{E}_0(\Omega)$ is defined as $$E(\varphi)=\int(-\varphi)(\triangle \varphi)^n.$$ We say $\mathcal{F}_\mu$ is proper if $\mathcal{F}_\mu\rightarrow+\infty$ whenever $E\rightarrow+\infty$. We say that a positive measure $\mu$ has \emph{finite quaternionic $p$-energy} if there exists a constant $C=C(p)>0$ such that 
$$ \int (-u)^pd\mu\leq CE_p(u)^{\frac{p}{n+p}},$$ for all $u\in \mathcal{E}_p(\Omega)$. Denote by $\mathcal{M}_p$ the set of all such positive measures. As proved in complex Hessian case (cf. Proposition 4.6 in \cite{lu}), $\mu\in\mathcal{M}_p$ if and only if $\mathcal{E}_p(\Omega)\subset L^p(\Omega,\mu)$.

\b{pro}\label{l6.1} (1) If $\{u_j\}\subset\mathcal{E}_1(\Omega)$ such that $\sup_j E(u_j)<+\infty$, then the function $u=(\sup_ju_j)^*$ is in $\mathcal{E}_1(\Omega)$.\\
(2) If $\{u_j\}\subset\mathcal{E}_1(\Omega)$ such that $\sup_j E(u_j)<+\infty$ and $u_j$ decreases to some function $u$ then the limit function $u$ is in $\mathcal{E}_1(\Omega)$.\\
(3) The functional $E:\mathcal{E}_1(\Omega)\rightarrow \mathbb{R}$ is lower semi-continuous.\\
(4) If $u,v\in \mathcal{E}_1(\Omega)$, then $$ E(u+v)^{\frac{1}{n+1}}\leq  E(u)^{\frac{1}{n+1}}+ E(v)^{\frac{1}{n+1}}.$$ Moreover, if $\mu\in \mathcal{M}_1$ then $\mathcal{F}_\mu$ is proper and convex.
\\
(5) If $\mu\in \mathcal{M}_1$, $u\in \mathcal{E}_1(\Omega)$ and $ \mathcal{E}_0(\Omega)\ni u_j\searrow u,$ then $\lim_{j\rightarrow\infty} \mathcal{F}_\mu(u_j)= \mathcal{F}_\mu(u).$\\
(6)  For $u,v\in \mathcal{E}_1(\Omega)$, $$\int_{\{u>v\}}(\triangle u)^n\leq\int_{\{u>v\}}(\triangle v)^n.$$
\e{pro}
\proof (1) Note that $u=(\sup_ju_j)^*$ is in $PSH^-(\Omega)$ by Proposition \ref{p1.1}. It follows from Theorem \ref{t1.1} that there exists a sequence $\{\varphi_j\}\subset \mathcal{E}_0(\Omega)\cap C(\Omega)$ such that $\varphi_j$ decreases to $u$. Thus $\varphi_j\geq u_j$. Since $\sup_j E(u_j)<+\infty$, we have $\sup_jE(\varphi_j)<+\infty$ by Corollary \ref{estimatecor}. Then $u\in \mathcal{E}_1(\Omega)$ by definition.

(2) As in the proof of (1), there exists a sequence $\{\varphi_j\}\subset \mathcal{E}_0(\Omega)\cap C(\Omega)$ such that $\varphi_j$ decreases to $u$, since $u\in PSH^-(\Omega)$ by Proposition \ref{p1.1}. Then the function $\psi_j:=\max\{u_j,\varphi_j\}\in\mathcal{E}_0(\Omega)$ by Corollary \ref{c4.3}, and $E(\psi_j)\leq E(u_j)$ by Corollary \ref{estimatecor}. It follows $u\in \mathcal{E}_1(\Omega)$.

(3) Let $\{u_j\}\subset \mathcal{E}_{1}(\Omega)$, $u\in\mathcal{E}_{1}(\Omega)$ be such that $u_j\rightarrow u$ in $L_{loc}^1(\Omega)$. By (1) above, the function $\varphi_j:=(\sup_{k\geq j}u_k)^*\in \mathcal{E}_1(\Omega)$, and $\varphi_j\searrow u$. By Proposition \ref{p4.33} and Corollary \ref{estimatecor}, $E(\varphi_j)\leq E(u_j)$ and $E(\varphi_j)\nearrow E(u)$. It follows that $\liminf_j E(u_j)\geq E(u)$.

(4) By the energy estimate (Theorem \ref{energy estimate}) we have 
\ben \b{aligned}E(u+v)=&\int (-u)(\triangle (u+v))^n +\int (-v)(\triangle (u+v))^n\\\leq& E(u)^{\frac{1}{n+1}}E(u+v)^{\frac{n}{n+1}}+E(v)^{\frac{1}{n+1}}E(u+v)^{\frac{n}{n+1}}. \e{aligned}\een 
It follows that $E^{\frac{1}{n+1}}$ is convex, thus $E$ is also convex. For $\mu\in \mathcal{M}_1$, there exists $C>0$ such that $\int (-u)d\mu\leq CE(u)^{\frac{1}{n+1}}$. So $\mathcal{F}_\mu(u)\geq \frac{1}{n+1}E(u)-CE(u)^{\frac{1}{n+1}},$ which implies that $\mathcal{F}_\mu$ is proper. 

(5) Note that $E(u_j)\nearrow E(u)$ by Proposition \ref{p4.33}. It follows from monotone convergence theorem and the fact $\mu\in \mathcal{M}_1$ that $\lim_j\mathcal{F}_\mu(u_j)$ is finite and depends only on $u$.

(6) As in the proof of Proposition \ref{t4.6}, we have $$\int_{\Omega}(-h)(\triangle \max\{u,v\})^n\leq\int_{\Omega}(-h)(\triangle u)^n,$$ for $h\in  \mathcal{E}_0(\Omega)\cap C(\Omega)$. All terms are finite for $h\in \mathcal{E}_0(\Omega)$, $u\in \mathcal{E}_1(\Omega)$, since 
$$\int_\Omega (-h)(\triangle u)^n\leq E(h)^{\frac{1}{n+1}} E(u)^{\frac{n}{n+1}}<+\infty.$$
Then it follows from Corollary \ref{l3.19'} that 
\beq\label{7.1}\b{aligned}
&\int_{\{u>v\}}(-h)(\triangle u)^n=\int_{\{u>v\}}(-h)(\triangle \max\{u,v\})^n\\\leq&\int_\Omega(-h)(\triangle \max\{u,v\})^n+\int_{\{u<v\}}h(\triangle \max\{u,v\})^n \\ \leq &\int_\Omega(-h)(\triangle v)^n+\int_{\{u<v\}}h(\triangle v)^n=\int_{\{u>v\}}(-h)(\triangle v)^n.
\e{aligned}\eeq
 Let $h\searrow -1$ to obtain the result.
\endproof

\subsection{The projection theorem}

Let $u:\Omega\rightarrow[-\infty,+\infty)$ be upper semi-continuous. Define the projection of $u$ on $\mathcal{E}_1(\Omega)$ by $$P(u)=\sup\{v\in\mathcal{E}_1(\Omega): v\leq u\}. $$ 
\b{lem} \label{l4.9'} Let $u:\Omega\rightarrow \mathbb{R}$ be continuous. Suppose that there exists $w\in \mathcal{E}_1(\Omega)$ such that $w\leq u$. Then $\int_{\{P(u)<u\}}(\triangle P(u))^n=0$. 
\e{lem}
\proof  Assume that $w$ is bounded on $\Omega$. By Choquet's lemma, one can find an increasing sequence $\{u_j\}\subset \mathcal{E}_1(\Omega)\cap L^\infty(\Omega)$ such that $u_j\leq u$ and $(\lim_j u_j)^*=P(u)$. Take $q_0\in \{P(u)<u\}.$ Since $u$ is continuous and $P$ is upper semi-continuous, there exist $\epsilon>0, r>0$ such that for each $q\in B=B(q_0,r)$, $P(u)(q)<u(q_0)-\epsilon<u(q)$, that is, $B\subset\{P(u)<u\}$. 

For fixed $j$, we can approximate $u_j|_{\partial B}$ from above by a sequence of continuous functions on $\partial B$. By \cite[Theorem 1.3]{alesker4}, we can find a function $\varphi_j\in PSH(B)$ such that $\varphi_j=u_j$ on $\partial B$, and $(\triangle \varphi_j)^n=0$ in $B$. It follows from the comparison principle that $u_j\leq\varphi_j$ in $B$. Define $$\psi_j=\left\{
      \begin{array}{ll}
        \varphi_j, & \text{in}~B, \\
        u_j, &  \text{in}~\Omega\backslash B.      \end{array}
    \right. $$ Noting that $\varphi_j=\max\{\varphi_j,u_j\}$ in $B$, we have $\varphi_j\in \mathcal{E}_1$, thus $\psi_j\in \mathcal{E}_1(\Omega)\cap L^\infty(\Omega)$. For each $q\in \partial B$, $\varphi_j(q)=u_j(q)\leq P(u)(q)\leq u(q_0)-\epsilon$. It follows that $\varphi_j\leq u(q_0)-\epsilon$ in $B$. So we have $u_j\leq \psi_j\leq u$ in $\Omega$. Therefore $(\lim_j \psi_j)^*=P(u)$. By the convergence result we established in \cite{wan4} (Lemma \ref{t3.2}), $(\triangle \psi_j)^n\rightarrow (\triangle P(u))^n$ weakly. It follows that $ (\triangle P(u))^n(B)\leq \liminf_j(\triangle \psi_j)^n(B)=0$.
\endproof

\b{lem}\label{l4.10'} Let $u\in \mathcal{E}_1(\Omega)$ and $v\in \mathcal{E}_1(\Omega)\cap C(\Omega)$. Define for $t<0$,
$$h_t=\frac{P(u+tv)-tv-u}{t}.$$ Then for each $0\leq k\leq n$, 
\beq\label{ht}\lim_{t\rightarrow0^-} \int_{\Omega} h_t(\triangle u)^k\wedge (\triangle P(u+tv))^{n-k}=0.\eeq In particular, $$\lim_{t\rightarrow0^-} \int_{\Omega} \frac{P(u+tv)-u}{t}(\triangle u)^k\wedge (\triangle P(u+tv))^{n-k}=\int_\Omega v(\triangle u)^n.$$
\e{lem}
\proof Noting that $u+tv\geq P(u+tv)\geq u$ for $t<0$, we have $0\leq h_t\leq -v$. It follows from $u\in \mathcal{E}_1(\Omega)$ that $P(u)=u$. Since $h_t$ is decreasing in $t$, for $s<0$ we have,
\ben\b{aligned}&\lim_{t\rightarrow0^-} \int_{\Omega} h_t(\triangle u)^k\wedge (\triangle P(u+tv))^{n-k}\leq \lim_{t\rightarrow0^-} \int_{\Omega} h_s(\triangle u)^k\wedge (\triangle P(u+tv))^{n-k}\\
&=\int_{\Omega} h_s(\triangle u)^k\wedge (\triangle P(u))^{n-k}=\int_{\Omega} h_s(\triangle u)^n=\int_{\{P(u+sv)-sv<u\}} h_s(\triangle u)^n \\&\leq \int_{\{P(u+sv)-sv<u\}} (-v)(\triangle u)^n.\e{aligned}\een
Take a decreasing sequence $\{u_k\}\subset \mathcal{E}_0(\Omega)\cap C(\Omega)$ such that $u_k\searrow u$ and 
$$ \int_{\{P(u+sv)-sv<u\}} (-v)(\triangle u)^n\leq C \int_{\{P(u_k+sv)-sv<u\}} (-v)(\triangle u)^n.$$ 
Note that $P(u_k+sv)-sv\in \mathcal{E}_1(\Omega)$. By Theorem \ref{t1.1}, the inequality (\ref{7.1}) in the proof of Proposition \ref{l6.1} (6) also holds for $h\in PSH^-(\Omega)$. So we have
\beq\label{7.2}\b{aligned}& \int_{\{P(u_k+sv)-sv<u\}} (-v)(\triangle u)^n\leq  \int_{\{P(u_k+sv)-sv<u\}} (-v)(\triangle (P(u_k+sv)-sv))^n\\ &\leq \int_{\{P(u_k+sv)-sv<u_k\}} (-v)(\triangle (P(u_k+sv)-sv))^n\\
&\leq  \|v\|\int_{\{P(u_k+sv)-sv<u_k\}} (\triangle (P(u_k+sv)-sv))^n.
\e{aligned}\eeq
Apply Lemma \ref{l4.9'} with continuous $u_k+sv$ to get $$\int_{\{P(u_k+sv)-sv<u_k\}} (\triangle (P(u_k+sv)))^n=0.$$
Therefore the right hand side of (\ref{7.2}) tends to $0$ as $s\rightarrow 0$. The proof is thus complete. \endproof

\b{thm}\label{t4.11'}Let $u\in \mathcal{E}_1(\Omega)$ and $v\in \mathcal{E}_1(\Omega)\cap C(\Omega)$. Then
$$ \frac{d}{dt}\Big|_{t=0} E(P(u+tv))=(n+1)\int_\Omega(-v)(\triangle u)^n.$$
\e{thm}
\proof Note that if $t>0$, $u+tv\in \mathcal{E}_1(\Omega)$, thus $P(u+tv)=u+tv$. It follows that $$\frac{d}{dt}\Big|_{t=0^+} E(u+tv)=(n+1)\int_\Omega(-v)(\triangle u)^n.$$
For $t<0$, since $P(u+tv),u\in \mathcal{E}_1(\Omega)$, we can integration by parts to get\ben\b{aligned}&\frac{1}{t}\left(\int -P(u+tv)(\triangle P(u+tv))^n-\int (-u)(\triangle u)^n\right)\\ =&\sum_{k=0}^n \int_\Omega \frac{u-P(u+tv)}{t} (\triangle u)^k\wedge (\triangle P(u+tv))^{n-k}\longrightarrow (n+1)\int_\Omega (-v)(\triangle u)^n, 
\e{aligned}\een as $t\nearrow0$, by Lemma \ref{l4.10'}. 
\endproof

\subsection{The quaternionic Monge-Amp\`{e}re equation}
In this section we combine all the results we established before to solve the quaternionic Monge-Amp\`{e}re equations. We follow the variational method in \cite{lu}.
\b{lem}\label{l4.12'} Assume that $\mu$ is a positive Radon measure such that $\mathcal{F}_\mu$ is proper and lower semicontinuous on $\mathcal{E}_1(\Omega)$. Then there exists $\varphi\in \mathcal{E}_1(\Omega)$ such that $\mathcal{F}_\mu(\varphi)=\inf_{\psi\in \mathcal{E}_1(\Omega)}\mathcal{F}_\mu(\psi).$
\e{lem}
\proof One can repeat the arguments in Lemma 4.12 in \cite{lu}.\endproof
\b{thm}\label{t4.13'} Let $\varphi\in \mathcal{E}_1(\Omega)$ and $\mu\in \mathcal{M}_1$. Then $$(\triangle\varphi)^n=\mu\Longleftrightarrow\mathcal{F}_\mu(\varphi)=\inf_{\psi\in \mathcal{E}_1(\Omega)}\mathcal{F}_\mu(\psi).$$
\e{thm}
\proof First assume that $(\triangle \varphi)^n=\mu$. By the energy estimate and Young's inequality, for each $\psi\in\mathcal{E}_1(\Omega)$ we have
\ben\b{aligned} \frac{1}{n+1}E(\psi)-\mathcal{F}_\mu(\psi)&=\int_{\Omega}(-\psi)(\triangle \varphi)^n\leq E(\psi)^{ \frac{1}{n+1}}E(\varphi)^{ \frac{n}{n+1}}\\
&\leq  \frac{1}{n+1}E(\psi)+ \frac{n}{n+1}E(\varphi)=\frac{1}{n+1}E(\psi)-\mathcal{F}_\mu(\varphi).\e{aligned}\een
It follows that $\mathcal{F}_\mu(\psi) \geq\mathcal{F}_\mu(\varphi)$, thus $\mathcal{F}_\mu(\varphi)=\inf_{\psi\in \mathcal{E}_1(\Omega)}\mathcal{F}_\mu(\psi).$

Now let $\psi\in \mathcal{E}_1(\Omega)\cap C(\Omega)$. Define $$ g(t)=\frac{1}{n+1}E(P(\varphi+t\psi))+L_\mu(\varphi+t\psi).$$ 
Note that $P(\varphi+t\psi)\leq \varphi+t\psi$ and $P(\varphi+t\psi)\in \mathcal{E}_1(\Omega)$. Assume that $\varphi$ minimizes $\mathcal{F}_\mu$ on $ \mathcal{E}_1(\Omega)$, then we have $$g(t)\geq \mathcal{F}_\mu(P(\varphi+t\psi))\geq \mathcal{F}_\mu(\varphi)=g(0),$$ i.e., $g$ obtains its minimum at $t=0$. It follows that $g'(0)=0$. Note that $\mu\in \mathcal{M}_1$, $$\frac{d}{dt}\Big|_{t=0}L_\mu (\varphi+t\psi)=L_\mu(\psi).$$ Then the result follows from Theorem \ref{t4.11'}.
\endproof

The following comparison principle follows from Proposition \ref{l6.1} (6) directly. 
\b{lem}\label{t3.26'} Let $u,v\in  \mathcal{E}_1(\Omega)$ with $(\triangle u)^n\geq (\triangle v)^n. $ Then $u\leq v$ in $\Omega$.
\e{lem}
We introduced in \cite{wan8} the quaternionic polar ($Q$-polar for short) set, which is the counterpart of pluripolar set in the complex case. See \cite{wan8} for detailed discussion of $Q$-polar set.
\b{lem}\label{l6.2} Let $\mu$ be a positive Radon measure in $\Omega$ such that $\mu(\Omega)<+\infty$ and $\mu$ does not charge $Q$-polar sets. Let $\{u_j\}\subset PSH^-(\Omega)$ converge to $u\in PSH^-(\Omega)$ in $L_{loc}^1$ and $\sup\int_\Omega(-u_j)^2d\mu<+\infty$. Then $\int_\Omega u_jd\mu\rightarrow \int_\Omega ud\mu$.
\e{lem}
\proof One can get this result by following the proof of the complex Hessian case in \cite[Lemma 4.2]{lu}. For convenience of readers we repeat it here. Since $\sup\int_\Omega(-u_j)^2d\mu<+\infty$, by Banach-Saks theorem there exists a subsequence $u_j$ such that $\varphi_N:=\frac{1}{N}\sum_{j=1}^N u_j$ converges in $L^2(\mu)$ and $\mu$-almost everywhere to $\varphi$. We also have $\varphi_N\rightarrow u$ in $L_{loc}^1$. Define $\psi_j:=(\sup_{k\geq j}\varphi_k)^*$ for each $j$. Then $\psi_j$ decreases to $u$ in $\Omega$. Note that the set $\{(\sup_{k\geq j} \varphi_k)^*>\sup_{k\geq j} \varphi_k\}$ is negligible, thus it is also $Q$-polar (cf. \cite[Theorem 1.2]{wan8}). Since $\mu$ does not charge on $Q$-polar set, we have $\psi_j=\sup_{k\geq j}\varphi_k$ $\mu$-almost everywhere. It follows that $\psi_j$ converges to $\varphi$ $\mu$-almost everywhere and $u=\varphi$ $\mu$-almost everywhere. So $\lim_j\int u_jd\mu=\lim_j\int \varphi_jd\mu=\int ud\mu$.
\endproof
\b{lem}\label{blocki}Let $\Omega$ be a bounded domain in $\mathbb{H}^n$. Suppose that $u,h,v_{1},\ldots,v_{n}$ are locally bounded plurisubharmonic functions
satisfying $u\leq h$ in $\Omega$, and
$\lim_{q\rightarrow\partial\Omega}(h(q)-u(q))=0$. Let
$v_{1},\ldots,v_{n}$ be nonpositive in
$\Omega$. Then
\begin{equation}\label{8.1}\int_{\Omega}(h-u)^{n}\triangle v_{1}\wedge\ldots\wedge\triangle v_{n}\leq
n!\|v_{1}\|_{\Omega}\cdots\|v_{n-1}\|_{\Omega}\int_{\Omega}|v_{n}|(\triangle u)^n.\end{equation}
\e{lem}
 \proof  We can assume that
$h=u$ in a neighborhood of $\partial\Omega$ and $h$ and $u$ are smooth. By hypothesis, we can let
$\|v_{i}\|_{\Omega}=1$. We want to show that for $p=2,\ldots,n$,
\begin{equation}\label{8.2}\begin{aligned}&\int_{\Omega}(h-u)^{p}(\triangle u)^{n-p}\wedge \triangle v_{n-p+1}\wedge\ldots\wedge \triangle v_{n}\\&\,\,\,\,\,\,\,\,\,\,\,\,\,\,\,\,\leq
p\int_{\Omega}(h-u)^{p-1}(\triangle u)^{n-p+1}\wedge \triangle v_{n-p+2}\wedge\ldots\wedge\triangle v_{n}.\end{aligned}\end{equation}
The left-hand side of (\ref{8.2}) is equal to
$$\int_{\Omega}v_{n-p+1}\triangle[(h-u)^{p}] \wedge(\triangle u)^{n-p}\wedge\triangle v_{n-p+2}\wedge\ldots\wedge\triangle v_{n}.$$
And
it follows from $-v_{n-p+1}=|v_{n-p+1}|\leq1$ that the right-hand
side of (\ref{8.2}) is greater than or equal to
$$p\int_{\Omega}(-v_{n-p+1})(h-u)^{p-1}(\triangle u)^{n-p+1}\wedge \triangle v_{n-p+2}\wedge\ldots\wedge\triangle v_{n}.$$
Note that $(\triangle u)^{n-p}\wedge \triangle v_{n-p+2}\wedge\ldots\wedge\triangle v_{n}$ is a closed positive $(2n-2)$-current. To show (\ref{8.2}),  it suffices to prove that for $p=2,\ldots,n$, \begin{equation}\label{8.22}-\triangle[(h-u)^{p}]\leq p(h-u)^{p-1}\triangle u,\end{equation} i.e., the difference between the right- and the left-hand side of (\ref{8.22}) is a positive $2$-form.
By (\ref{2.228}) we
have\begin{equation}\label{8.4}\begin{aligned}&\triangle[(h-u)^{p}]=d_0d_1[(h-u)^{p}]=d_0\left[p(h-u)^{p-1}d_1(h-u)\right]\\
=&p(p-1)(h-u)^{p-2}d_0(h-u)\wedge d_1(h-u)+p(h-u)^{p-1}d_0d_1(h-u)\\
\geq&p(p-1)(h-u)^{p-2}d_0(h-u)\wedge d_1(h-u)-p(h-u)^{p-1}d_0d_1u\\
\geq&-p(h-u)^{p-1}\triangle u,\end{aligned}\end{equation}where
the first inequality follows from the positivity of
$\triangle h=d_0d_1h$ and the last inequality follows from the positivity of
$d_0(h-u)\wedge d_1(h-u)$. Then we get (\ref{8.2}).
And \begin{equation*}\begin{aligned}\int_{\Omega}(h-u)^{n}\triangle v_{1}\wedge\ldots\wedge\triangle v_{n}
\leq n!\int_{\Omega}(h-u)(\triangle u)^{n-1}\wedge\triangle v_n\\
=n!\int_{\Omega}(-v_{n})(\triangle u)^{n-1}\wedge\triangle(u-h)
\leq n!\int_{\Omega}|v_{n}|(\triangle u)^n.\end{aligned}\end{equation*}\par
\endproof

\b{thm}\label{t4.14'}Let $\mu$ be a positive Radon measure in $\mathcal{M}_1$. Then there exists a unique $u\in\mathcal{E}_1(\Omega)$ such that $(\triangle u)^n=\mu$.
\e{thm}
\proof The uniqueness follows from Lemma \ref{t3.26'} above. First suppose that $\mu$ has compact support $K\Subset \Omega$. Denote by $u_K:=u_{K,\Omega}^*$ the regularized relatively extremal function of $K$ given by (\ref{u*}). Then $u_K\leq u_L$ for each compact $L\subset K$, thus $E(u_L)\leq E(u_K)$. Define
$$ \mathcal{M}=\left\{\nu>0: \text{supp}~\nu\subset K, ~\int_\Omega(-\varphi)^2d\nu\leq CE(\varphi)^{\frac{2}{n+1}},~\text{for}~\forall \varphi\in \mathcal{E}_1(\Omega)\right\},$$ where $C$ is chosen such that $C>2E(u_K)^{\frac{n-1}{n+1}}$. We have $(\triangle u_L)^n\in \mathcal{M}$ for each compact $L\subset K$. This is because 
\ben\b{aligned} \int(-\varphi)^2(\triangle u_L)^n&\leq 2\|u_L\|_\Omega\int_\Omega (-\varphi)\triangle \varphi\wedge (\triangle u_L)^{n-1}\\&\leq 2E(\varphi)^{\frac{2}{n+1}}E(u_L)^{\frac{n-1}{n+1}}\leq 2E(\varphi)^{\frac{2}{n+1}}E(u_K)^{\frac{n-1}{n+1}}<CE(\varphi)^{\frac{2}{n+1}},
\e{aligned}\een by (\ref{8.2}) in Lemma \ref{blocki} and energy estimate. Since integration by parts and convergence result are valid on $ \mathcal{E}_1(\Omega)$ by Theorem \ref{t3.15} and Theorem \ref{t4.4}, the estimate (\ref{8.2}) is applicable for $ \varphi\in \mathcal{E}_1(\Omega)$.

Now one can repeat the arguments in \cite{lu} to complete the proof of this theorem. We give a word-by-word copy here for the sake of completeness. 
Let $S:=\sup\{\nu(\Omega):\nu\in \mathcal{M}\}. $ We have $S<+\infty$. Indeed, since $\Omega$ is hyperconvex, we can find $h\in PSH^-(\Omega)\cap C(\overline{\Omega})$ such that $h<-1$ on $K\Subset \Omega$, and $$\nu(K)\leq \int_K(-h)^2d\nu\leq CE(h)^{{\frac{2}{n+1}}}.$$

Fix $\nu_0\in \mathcal{M}$. Define 
$$ \mathcal{M}'=\left\{\nu:\text{supp}\nu\subset K,\int_\Omega(-\varphi)^2d\nu\leq \left[\frac{C}{S}+\frac{C}{\nu_0(\Omega)}\right]E(\varphi)^{\frac{2}{n+1}},~\text{for}~\forall \varphi\in \mathcal{E}_1(\Omega)\right\}.$$
Then for each $ \nu\in\mathcal{M}$ and $\varphi\in\mathcal{E}_1(\Omega)$, one can check that $$\frac{(S-\nu(\Omega))\nu_0+\nu_0(\Omega)\nu}{S\nu_0(\Omega)}\in\mathcal{M}'.$$ Therefore $\mathcal{M}'$ is non-empty convex and weakly compact. By a generalized Radon-Nykodim Theorem, there exists a positive $\nu\in\mathcal{M}'$ and a positive function $f\in L^1(\nu)$ such that $\mu=fd\nu+\sigma$ with $\sigma$ orthogonal to $\mathcal{M}'$. Since $(\triangle u_L)^n\in\mathcal{M}$ for each compact $L\subset K$, each measure orthogonal to $\mathcal{M}'$ must be supported in some $Q$-polar set. Since $\mu\in\mathcal{M}_1$ does not charge $Q$-polar sets, we have $\sigma\equiv 0$, thus $\mu=fd\nu$ with $f\in L^1(\nu)$, $\nu\in\mathcal{M}'$.

For each $j\in\mathbb{N}$, set $\mu_j=\min(f,j)\nu$. By Proposition \ref{l6.1} and Lemma \ref{l6.2} we see that $\mathcal{F}_{\mu_j}$ is proper and lower semi-continuous on $\mathcal{E}_1(\Omega)$. Therefore by Lemma \ref{l4.12'} and Theorem \ref{t4.13'}, there exists $u_j\in\mathcal{E}_1(\Omega)$ such that $(\triangle u_j)^n=\mu_j$. By Lemma \ref{t3.26'}, $u_j$ decreases to a function $u\in PSH^-(\Omega)$, which satisfies $(\triangle u)^n=\mu$. The existence is proved. 

By Proposition \ref{l6.1}, $u$ must be in $\mathcal{E}_1(\Omega)$. Indeed, since $\mu\in\mathcal{M}_1$, $$E(u_j)=\int_\Omega(-u_j)(\triangle u_j)^n=\int_\Omega(-u_j)d\mu_j\leq \int_\Omega(-u_j)d\mu\leq CE(u_j)^{\frac{1}{n+1}}. $$
This means that $\sup_jE(u_j)<+\infty$. 

When $\mu$ does not have compact support, it suffices to approximate $\Omega$ by a sequence of compact subsets $\{K_j\}$ and consider $\mu_j=\chi_{K_j}d\mu$.
\endproof

\b{lem}\label{l5.1'} Let $\mu$ be a positive Radon measure with $\mu(\Omega)<+\infty$. Assume that $\psi$ is a bounded PSH function in $\Omega$ such that $(\triangle \psi)^n\geq \mu$. Then there exists a unique $\varphi\in \mathcal{E}_0(\Omega)$ such that $(\triangle \varphi)^n=\mu$.
\e{lem}
\proof Assume that $-1\leq \psi\leq 0$. Let $h\in  \mathcal{E}_0(\Omega)$ be the exhaustion function of $\Omega$. Let $h_j=\max\{\psi,jh\}$ and $A_j=\{q\in \Omega:jh<-1\}$. Note that $\chi_{A_j}\mu\in \mathcal{M}_1$. By Theorem \ref{t4.14'}, for each $j$, there exists $\varphi_j\in   \mathcal{E}_1(\Omega)$ such that $(\triangle \varphi_j)^n=\chi_{A_j}\mu.$ 

On $A_j$, $(\triangle \varphi_j)^n=\mu\leq (\triangle \psi)^n$, thus $\varphi_j\geq \psi=h_j$ on $A_j$ by Lemma \ref{t3.26'}. Noting that outside $A_j$, $(\triangle \varphi_j)^n=0\leq (\triangle h_j)^n$, we have $0\geq\varphi_j\geq h_j\geq \psi$ on $\Omega$. Thus $\varphi_j\in\mathcal{E}_0(\Omega)$. Again by Lemma \ref{t3.26'}, $\varphi_j$ decreases to some $\varphi\in \mathcal{E}_0(\Omega)$ and $\varphi$ satisfies $(\triangle \varphi)^n=\mu$.
\endproof

\b{pro}\label{t5.2'} Let $u,v\in \mathcal{E}_p(\Omega)$, $p>1$, then $$\int_{\{u>v\}}(\triangle u)^n\leq \int_{\{u>v\}}(\triangle v)^n.$$
\e{pro}
\proof Let $h\in \mathcal{E}_0(\Omega)\cap C(\Omega)$. First assume that $v$ is bounded and vanishes on $\partial\Omega$. Let $K_j$ be an exhaustion sequence of compact subsets of $\Omega$. By Lemma \ref{l5.1'}, there exists $v_j\in\mathcal{E}_0(\Omega)$ such that $(\triangle v_j)^n=\chi_{K_j}(\triangle v)^n$. Then $v_j$ decreases to $v$ by Lemma \ref{t3.26'}. Since $\int_\Omega(-h)(\triangle v_j)^n$ is finite, by Corollary \ref{l3.19'} we can get 
$$\int_{\{u>v_j\}}(-h)(\triangle u)^n\leq \int_{\{u>v_j\}}(-h)(\triangle v_j)^n=\int_{\{u>v_j\}\cap K_j}(\triangle v)^n,$$ as in the proof of Proposition \ref{l6.1} (6). Let $j\rightarrow+\infty$ to get $$\int_{\{u>v\}}(-h)(\triangle u)^n\leq \int_{\{u>v\}}(-h)(\triangle v)^n.$$
Now it suffices to remove the assumption on $v$ as in the proof of Theorem 5.2 in \cite{lu}. \endproof

\b{pro}\label{t5.3'}Let $\mu$ be a positive measure in $\Omega$ which does not charge $Q$-polar sets. Then there exists $\varphi\in \mathcal{E}_0(\Omega)$ and $0\leq f\in L_{loc}^1((\triangle \varphi)^n)$ such that $\mu=f(\triangle \varphi)^n$.
\e{pro}
\proof Assume first that $\mu$ has compact support. Since $\mu$ does not charge $Q$-polar sets, the arguments of the proof of Theorem \ref{t4.14'} can be applied to find $u\in \mathcal{E}_1(\Omega)$ and $0\leq f\in L^1((\triangle u)^n)$ such that $\mu=f(\triangle u)^n$ and $\text{supp} (\triangle u)^n\Subset \Omega$. Define $\psi=(-u)^{-1}$. Since $\gamma(t):=(-t)^{-1}$ is a convex and increasing function of $t\in (-\infty,0)$, $\psi$ is also PSH by Proposition \ref{p1.1}. Note that $u$ is upper semi-continuous, thus is local bounded from above, so does $\psi$. Hence $\psi\in L_{loc}^\infty(\Omega)$. Since $\text{supp} (\triangle u)^n\Subset \Omega$, we can modify $\psi$ in a neighborhood of $\partial\Omega$ such that $\psi\in \mathcal{E}_0(\Omega)$. By direct computation we have $ (\triangle \psi)^n\geq (-u)^{-2n}(\triangle u)^n$ (as in (\ref{8.4})). By Lemma \ref{l5.1'} we can find $\varphi\in \mathcal{E}_0(\Omega)$ such that $(-u)^{-2n}(\triangle u)^n=(\triangle \varphi)^n.$ Therefore $\mu=f(\triangle u)^n=f(-u)^{2n}(\triangle \varphi)^n$. 

When $\mu$ does not have compact support, it suffices to approximate $\Omega$ by an exhaustive sequence of compact subsets $\{K_j\}$ and consider $\chi_{K_j}\mu=f_j(\triangle u_j)^n.$ \endproof

\b{thm} \label{t5.4'}Let $\mu$ be a positive Radon measure on $\Omega$ such that $\mathcal{E}_p(\Omega)\subset L^p(\Omega,\mu)$, $p\geq1$. Then there exists a unique $\varphi\in \mathcal{E}_p(\Omega)$ such that $(\triangle \varphi)^n=\mu$.
\e{thm}
\proof The uniqueness follows from Proposition \ref{t5.2'}. Since $\mu$ does not charge $Q$-polar sets, by Proposition \ref{t5.3'} there exists $u\in \mathcal{E}_0(\Omega)$ and $0\leq f\in L^1_{loc}((\triangle u)^n)$ such that $\mu=f(\triangle u)^n$. For each $j$ let $\mu_j:=\min(f,j)(\triangle u)^n$. By Lemma \ref{l5.1'} we can find $\varphi_j\in\mathcal{E}_0(\Omega)$ such that $(\triangle \varphi_j)^n=\mu_j$. Since $\mu \in \mathcal{M}_p$, we have $\sup_jE_p(\varphi_j)<+\infty$. It follows from Proposition \ref{t5.2'} and definition of $\mathcal{E}_p(\Omega)$ that $\varphi_j$ decreases to some $\varphi\in\mathcal{E}_p(\Omega)$ and $\varphi$ satisfies $(\triangle \varphi)^n=\mu$.
\endproof 

\b{lem}\label{p6.1'}Let $u,v\in\mathcal{E}_p(\Omega)$ and $p\geq1$. Then there exist two sequences $\{u_j\},\{v_j\}\subset \mathcal{E}_0(\Omega)$ decreasing to $u,v$ respectively such that $$\lim_{j\rightarrow+\infty}\int_\Omega (-u_j)^p(\triangle v_j)^n=\int_\Omega (-u)^p(\triangle v)^n.$$
\e{lem}
\proof Since $u\in\mathcal{E}_p(\Omega)$, there exists a sequence $\{u_j\}\subset \mathcal{E}_0(\Omega)$ decreasing to $u$ such that $$\sup_j\int_\Omega (-u_j)^p(\triangle u_j)^n<+\infty.$$ Noting that $(\triangle v)^n$ does not charge on $Q$-polar sets, we can get from Proposition \ref{t5.3'} that $(\triangle v)^n=f(\triangle \psi)^n$ with $\psi\in\mathcal{E}_0(\Omega)$, $0\leq f\in L^1_{loc}((\triangle \psi)^n)$. Then by Lemma \ref{l5.1'} we can find a sequence $\{v_j\}\subset \mathcal{E}_0(\Omega)$ such that $(\triangle v_j)^n=\min(f,j)(\triangle \psi)^n$. It follows from the comparison principle that $v_j$ decreases to some function $\varphi\in \mathcal{E}_p(\Omega)$ satisfying $(\triangle v)^n=(\triangle \varphi)^n$. Hence we have $v\equiv \varphi$. Therefore $$\int_\Omega (-u)^p(\triangle v)^n=\lim_{j\rightarrow+\infty}\int_\Omega (-u_j)^p\min(f,j)(\triangle \psi)^n=\lim_{j\rightarrow+\infty}\int_\Omega (-u_j)^p(\triangle v_j)^n.$$
\endproof

\emph{Proof of Theorem \ref{t6.2'}} Assume that $\mu=(\triangle \varphi)^n$ with $\varphi\in\mathcal{E}_p(\Omega)$ and $\psi$ is another function in $\mathcal{E}_p(\Omega)$. It follows from Lemma \ref{p6.1'} that there exist $\{\varphi_j\},\{\psi_j\}\subset \mathcal{E}_0(\Omega)$ decreasing to $\varphi,\psi$ such that $\sup_j\int_\Omega (-\varphi_j)^p(\triangle \varphi_j)^n<+\infty,$ $\sup_j\int_\Omega (-\psi_j)^p(\triangle \psi_j)^n<+\infty$ and $$\int_\Omega (-\psi)^p(\triangle \varphi)^n=\lim_{j\rightarrow+\infty}\int_\Omega (-\psi_j)^p(\triangle \varphi_j)^n<+\infty$$ by energy estimate. So we have $\psi\in L^p(\Omega,\mu)$. Then it suffices to apply Theorem \ref{t5.4'}. \qed

\begin{Acknw}
This work is supported by National Nature Science Foundation in China (No. 11401390; No.
11571305).
\end{Acknw}
\section*{References}
\bibliography{mybibfile}

\begin{thebibliography}{10}

\bibitem{alesker1}
Semyon Alesker.
\newblock Non-commutative linear algebra and plurisubharmonic functions of
  quaternionic variables.
\newblock {\em Bull. Sci. Math.}, 127(1):1--35, 2003.

\bibitem{alesker4}
Semyon Alesker.
\newblock Quaternionic {M}onge-{A}mp\`ere equations.
\newblock {\em J. Geom. Anal.}, 13(2):205--238, 2003.

\bibitem{alesker2}
Semyon Alesker.
\newblock Pluripotential theory on quaternionic manifolds.
\newblock {\em J. Geom. Phys.}, 62(5):1189--1206, 2012.

\bibitem{alesker9}
Semyon Alesker.
\newblock Solvability of the quaternionic {M}onge-{A}mp\`ere equation on
  compact manifolds with a flat hyper{K}\"ahler metric.
\newblock {\em Adv. Math.}, 241:192--219, 2013.

\bibitem{alesker6}
Semyon Alesker and Egor Shelukhin.
\newblock On a uniform estimate for the quaternionic {C}alabi problem.
\newblock {\em Israel J. Math.}, 197(1):309--327, 2013.

\bibitem{alesker7}
Semyon Alesker and Misha Verbitsky.
\newblock Quaternionic {M}onge-{A}mp\`ere equation and {C}alabi problem for
  {HKT}-manifolds.
\newblock {\em Israel J. Math.}, 176:109--138, 2010.

\bibitem{bed1990}
Eric Bedford.
\newblock Survey of pluri-potential theory.
\newblock In {\em Several complex variables ({S}tockholm, 1987/1988)},
  volume~38 of {\em Math. Notes}, pages 48--97. Princeton Univ. Press,
  Princeton, NJ, 1993.

\bibitem{beddirichlet}
Eric Bedford and B.~A. Taylor.
\newblock The {D}irichlet problem for a complex {M}onge-{A}mp\`ere equation.
\newblock {\em Invent. Math.}, 37(1):1--44, 1976.

\bibitem{bed}
Eric Bedford and B.~A. Taylor.
\newblock A new capacity for plurisubharmonic functions.
\newblock {\em Acta Math.}, 149(1-2):1--40, 1982.

\bibitem{guedj-variational}
Robert~J. Berman, Sébastien Boucksom, Vincent Guedj, and Ahmed Zeriahi.
\newblock A variational approach to complex monge-ampère equations.
\newblock {\em Publications Mathématiques De Lihés}, 117(1):179--245, 2013.

\bibitem{cegrell1986}
Urban Cegrell.
\newblock Sums of continuous plurisubharmonic functions and the complex
  {M}onge-{A}mp\`ere operator in {${\bf C}^n$}.
\newblock {\em Math. Z.}, 193(3):373--380, 1986.

\bibitem{cegrell1998}
Urban Cegrell.
\newblock Pluricomplex energy.
\newblock {\em Acta Math.}, 180(2):187--217, 1998.

\bibitem{cegrell2004}
Urban Cegrell.
\newblock The general definition of the complex {M}onge-{A}mp\`ere operator.
\newblock {\em Ann. Inst. Fourier (Grenoble)}, 54(1):159--179, 2004.

\bibitem{cegrell2}
Urban Cegrell.
\newblock Approximation of plurisubharmonic functions in hyperconvex domains.
\newblock In {\em Complex analysis and digital geometry}, volume~86 of {\em
  Acta Univ. Upsaliensis Skr. Uppsala Univ. C Organ. Hist.}, pages 125--129.
  Uppsala Universitet, Uppsala, 2009.

\bibitem{demailly85}
Jean-Pierre Demailly.
\newblock Mesures de {M}onge-{A}mp\`ere et caract\'erisation g\'eom\'etrique
  des vari\'et\'es alg\'ebriques affines.
\newblock {\em M\'em. Soc. Math. France (N.S.)}, (19):124, 1985.

\bibitem{Demailly1991}
Jean-Pierre Demailly.
\newblock Potential theory in several complex variables.
\newblock 1991.
\newblock available online at:
  http://www-fourier.ujf-grenoble.fr/~demailly/documents.html.

\bibitem{kang}
Qianqian Kang and Wei Wang.
\newblock On {P}enrose integral formula and series expansion of {$k$}-regular
  functions on the quaternionic space {$\Bbb{H}^n$}.
\newblock {\em J. Geom. Phys.}, 64:192--208, 2013.

\bibitem{hiep}
Nguyen~Van Khue and Pham~Hoang Hiep.
\newblock A comparison principle for the complex {M}onge-{A}mp\`ere operator in
  {C}egrell's classes and applications.
\newblock {\em Trans. Amer. Math. Soc.}, 361(10):5539--5554, 2009.

\bibitem{kiselman}
Christer~O. Kiselman.
\newblock Sur la d\'efinition de l'op\'erateur de {M}onge-{A}mp\`ere complexe.
\newblock In {\em Complex analysis ({T}oulouse, 1983)}, volume 1094 of {\em
  Lecture Notes in Math.}, pages 139--150. Springer, Berlin, 1984.

\bibitem{klimek}
Maciej Klimek.
\newblock {\em Pluripotential theory}, volume~6 of {\em London Mathematical
  Society Monographs. New Series}.
\newblock The Clarendon Press, Oxford University Press, New York, 1991.
\newblock Oxford Science Publications.

\bibitem{lu}
Chinh~H. Lu.
\newblock A variational approach to complex {H}essian equations in
  {$\Bbb{C}^n$}.
\newblock {\em J. Math. Anal. Appl.}, 431(1):228--259, 2015.

\bibitem{lu-degenerate}
Chinh~H Lu and Van~Dong Nguyen.
\newblock Degenerate complex hessian equations on compact k\.
\newblock {\em Mathematics}, 64(6), 2015.

\bibitem{persson}
Leif Persson.
\newblock A {D}irichlet principle for the complex {M}onge-{A}mp\`ere operator.
\newblock {\em Ark. Mat.}, 37(2):345--356, 1999.

\bibitem{sibony85}
Nessim Sibony.
\newblock Quelques probl\`emes de prolongement de courants en analyse complexe.
\newblock {\em Duke Math. J.}, 52(1):157--197, 1985.

\bibitem{wan7}
Dongrui Wan.
\newblock The continuity and range of the quaternionic {M}onge-{A}mp\`{e}re
  operator on quaternionic space.
\newblock {\em Math. Z.}, 285:461–--478, 2017.

\bibitem{wan8}
Dongrui Wan and Qianqian Kang.
\newblock Potential theory for quaternionic plurisubharmonic functions.
\newblock {\em Michigan Math. J.}, 66:3–--20, 2017.

\bibitem{wan6}
Dongrui Wan and Wei Wang.
\newblock Viscosity solutions to quaternionic {M}onge-{A}mp\`ere equations.
\newblock {\em Nonlinear Anal.}, 140:69--81, 2016.

\bibitem{wan3}
Dongrui Wan and Wei Wang.
\newblock On quaternionic {M}onge-{A}mp\`{e}re operator, closed positive
  currents and {L}elong-{J}ensen type formula on quaternionic space.
\newblock {\em Bull. Sci. math.}, 141:267–--311, 2017.

\bibitem{wan4}
Dongrui Wan and Wenjun Zhang.
\newblock Quasicontinuity and maximality of quaternionic plurisubharmonic
  functions.
\newblock {\em J. Math. Anal. Appl.}, 424:86--103, 2015.

\bibitem{Wang3}
Wei Wang.
\newblock On non-homogeneous {C}auchy-{F}ueter equations and {H}artogs'
  phenomenon in several quaternionic variables.
\newblock {\em J. Geom. Phys.}, 58(9):1203--1210, 2008.

\bibitem{Wang}
Wei Wang.
\newblock The {$k$}-{C}auchy-{F}ueter complex, {P}enrose transformation and
  {H}artogs phenomenon for quaternionic {$k$}-regular functions.
\newblock {\em J. Geom. Phys.}, 60(3):513--530, 2010.

\bibitem{wang1}
Wei Wang.
\newblock On the optimal control method in quaternionic analysis.
\newblock {\em Bull. Sci. Math.}, 135(8):988--1010, 2011.

\bibitem{Wang2}
Wei Wang.
\newblock The tangential {C}auchy-{F}ueter complex on the quaternionic
  {H}eisenberg group.
\newblock {\em J. Geom. Phys.}, 61(1):363--380, 2011.

\bibitem{cegrell2012}
Per Åhag, Urban Cegrell, and Rafał Czyż.
\newblock On dirichlet's principle and problem.
\newblock 110(2):235--250, 2012.

\end{thebibliography}
\bibliographystyle{plain}
 \end{document}